\documentclass[a4paper,11pt]{article}
\usepackage{amssymb,amsmath,amsthm,amsfonts}
\usepackage{bookmark}
\usepackage{mathrsfs}
\usepackage{multirow}
\usepackage{graphicx}
\usepackage{authblk}
\usepackage{indentfirst}
\usepackage{multicol}
\usepackage{tabu}
\usepackage{url}
\usepackage{fancyhdr}
\usepackage[numbers]{natbib}
\usepackage[all]{xy}
\usepackage{mathtools}
\usepackage[english]{babel}
\usepackage{colortbl}
\usepackage{caption}
\usepackage{hyperref}\usepackage{subcaption}
\usepackage{tikz}\usepackage{float}

\newtheorem{theorem}{Theorem}[section]
\newtheorem{proposition}[theorem]{Proposition}
\newtheorem{lemma}[theorem]{Lemma}

\newtheorem{definition}{Definition}[section]
\newtheorem{example}{Example}[section]

\newcommand{\im}{{\mathrm{im}\hspace{0.1em}}}

\newcommand{\rank}{{\mathrm{rank}\hspace{0.1em}}}
\definecolor{myred}{RGB}{180,0.00,0.00}
\definecolor{myblue}{RGB}{0.00,0.00,180}
\usepackage{xr}
\makeatletter
    \newcommand*{\addFileDependency}[1]{
    \typeout{(#1)}
    \@addtofilelist{#1}
    \IfFileExists{#1}{}{\typeout{No file #1.}}
    }
\makeatother

\title{Persistent interaction topology in data analysis}

\author[1,2]{Jian Liu \thanks{Corresponding author: liujia53@msu.edu}}
\author[2]{Dong Chen}
\author[2,3,4]{Guo-Wei Wei \thanks{Corresponding author: weig@msu.edu}}
\affil[1]{Mathematical Science Research Center, Chongqing University of Technology, Chongqing 400054, China}
\affil[2]{Department of Mathematics, Michigan State University, MI 48824, USA}
\affil[3]{Department of Electrical and Computer Engineering, Michigan State University, MI 48824, USA}
\affil[4]{Department of Biochemistry and Molecular Biology, Michigan State University, MI 48824, USA}

\makeatletter
    \renewcommand*{\@fnsymbol}[1]{\ensuremath{\ifcase#1\or \dagger\or *\or *\or
   \mathsection\or \else\@ctrerr\fi}}
\makeatother
\date{}

\begin{document}
    \maketitle

    \paragraph{Abstract}
    Topological data analysis, as a tool for extracting topological features and characterizing geometric shapes, has experienced significant development across diverse fields. Its key mathematical techniques  include persistent homology and the recently developed persistent Laplacians. However, classic mathematical models like simplicial complexes often struggle to provide a localized topological description for interactions or individual elements within a complex system involving a specific set of  elements. In this work, we introduce  persistent interaction homology and persistent interaction Laplacian that emphasize individual interacting elements in the system. We demonstrate the stability of persistent interaction homology as a persistent module. Furthermore, for a finite discrete set of points in the Euclidean space, we provide the construction of persistent interaction Vietoris-Rips complexes and compute their interaction homology and interaction Laplacians. The proposed methods hold significant promise for analyzing heterogeneously interactive  data and emphasizing specific elements in data. Their utility for data science is demonstrated with applications to molecules.

    \paragraph{Keywords}
     Topological data analysis, interaction homology, interaction Laplacian, topological persistence, interaction Vietoris-Rips complex.

\footnotetext[1]
{ {\bf 2020 Mathematics Subject Classification.}  	Primary  55N31; Secondary 62R40,  68T09.
}

    \newpage

\section{Introduction}\label{section:introduction}

In the past decade, topological data analysis (TDA), as a tool capable of capturing the overall and topological features of data, has been widely applied in various fields, including biology, materials science, physics, and computer science. Its key idea is to grasp the various topological invariants of data across different scales. Persistent homology \cite{carlsson2004persistence,edelsbrunner2008persistent,zomorodian2004computing}, as a major tool in TDA, can identify persistent topological structures in a dataset, providing insights into important features that cannot be obtained from traditional data analytic tools. It not only discovers voids and loops in the data but also quantitatively describes the persistence of these topological features and formalizes them into mathematical concepts, thus providing a powerful tool and framework for data science. However, persistent homology has many limitations, such as its indifference to homotopic changes during the filtration and the specific elements and the number of nodes in a loop, which are crucial for complex data.   These challenges were addressed in part by persistent Hodge Laplacians  (or evolutionary de Rham-Hodge theory   \cite{chen2019evolutionary}) and persistent combinatorial Laplacian (or persistent Laplacian \cite{wang2019persistent}) in 2019. The kernel space of the persistent Laplacian is referred to as the persistent harmonic space, which is isomorphic to   persistent homology, implying that the persistent Laplacian encapsulates more information than persistent homology due to its non-harmonic spectrum. Moreover, the stability of the persistent Laplacian ensures its robustness as  features \cite{liu2023algebraic}. Furthermore, researchers have investigated persistent Laplacians on different topological objects   \cite{chen2023persistent,shen2023persistent,wang2023persistent,wei2021persistent}, computational algorithms \cite{memoli2022persistent}, and applications \cite{qiu2023persistent,wee2022persistent,wei2023persistent}.
Recently, persistent Dirac operators have also been developed to achieve similar goals \cite{ameneyro2024quantum,suwayyid2024persistent,wee2023persistent}.
While these new topological formulations provide new perspectives on topological invariants and geometric features of the data, the element-specific formulation proposed in an early work \cite{cang2017topologynet} is still required in practical applications \cite{qiu2023persistent,wee2022persistent,wei2023persistent}.

To further advance TDA, it is imperative to explore alternative topological formulations beyond homology, Laplacian, and Dirac and extract element-specific topological information in contrast to the usual global topological information.
Interaction refers to the dynamic exchange and mutual influence between various components within a system. It encompasses the ways in which these components affect each other's behaviors, properties, and functions, often resulting in emergent phenomena or system-level outcomes that cannot be fully explained by considering individual components in isolation. Interactions can occur at different scales, from interactions between fundamental particles in physics to complex molecular interactions in chemistry, and even to interactions between organisms in biological systems. Understanding interactions is essential for comprehending the dynamics and behavior of systems across various disciplines.

There are numerous methods to describe interactions within systems, with topology being considered particularly adept at capturing the essence of these interactions. Among these, \v{C}ech cohomology \cite{bott1982differential,wells1980differential} stands out as a prominent invariant, focusing on the intersections between open sets in an open cover. A related concept is the nerve complex \cite{alexandroff1928allgemeinen,eilenberg2015foundations}, which provides a topological representation of the relationships between various components within a system. However, both \v{C}ech cohomology and nerve complexes primarily address a form of binary judgment regarding the interactions between different parts of the system: whether they intersect or have some form of association. This characterization, to some extent, is somewhat crude.
In this work, we aim to study the interactions between topological spaces and  characterize these interactions using topological invariants. Specifically, we seek to use topological invariants to capture the structure and dynamics of interacting systems, and to gain insights into their underlying geometric and topological properties for specific elements. One potential approach is the interaction topology developed in \cite{liu2023interaction}.
The notion of interaction cohomology, initially proposed by Oliver Knill as a method for decoding the Wu characteristic \cite{knill2018cohomology}, stands as a pivotal concept in this endeavor.
As the word ``interaction'' suggests, interaction homology reflects the interaction or intersection relationship between simplicial complexes. Interaction (co)homology is derived from the interaction chain complex, which is constructed from a family of tuples of intersection simplices, or ``intersection cells''. Each interaction simplex indicates a interaction among some simplices. By employing interaction topology, one can obtain a different topological description of a cover, which better characterizes the behavior of interactions or the elements of interest. Hence, studying the interaction homotopy and homology of datasets provides us with a topological intuition about the interactions between systems or the points of interest.

In this work, we introduce two new methods, namely persistent interaction homology (PIH) and persistent interaction Laplacian (PIL), built upon the theoretical framework of interaction topology. PIH and PILs provide new topological characteristics, emphasizing spatial interaction relationships or a specific set of elements.  Unlike traditional methods such as persistent homology and persistent Laplacians, which view the space or point set as a whole and calculate Betti numbers and Laplacian matrices accordingly, PIH and PILs focus on describing the interactions and relationships among various elements of the space or point set. PIH and PILs are better than traditional topological methods at capturing the interactions between different parts,   leading to an emphasis of certain elements of interest in complex systems.
Additionally, we demonstrate the stability of persistent interaction homology as a persistent invariant. Furthermore, we introduce interaction Vietoris-Rips complexes, outlining their construction and computation on point sets. Through concrete examples, we illustrate how this new approach provides information on Betti numbers and spectral gaps of Laplacians.


The rest of the paper is organized as follows. In the next section, we review interaction topology, including concepts such as interaction homology and interaction Laplacian. Section \ref{section:PIH} introduces persistent interaction homology and investigates its stability. Additionally, we discuss  interaction Vietoris-Rips complexes and persistent interaction Laplacians. Section \ref{section:application} presents the applications of persistent interaction homology and persistent interaction Laplacian. Finally, we conclude the paper with a summary.

\section{Interaction topology}\label{section:interaction}
Interaction spaces are introduced to describe the interactions among different spaces within a complex system. The interactions between spaces play crucial roles in diverse fields. Our method is built on the topology of the category of interaction spaces \cite{liu2023interaction}. In this section, we will consider the interaction homology and the interaction Laplacians for interaction simplicial complexes.
To enhance the readability of the paper, this section provides several specific examples of computing interaction homology and interaction Laplacian.
From now on, the ground ring is assumed to be a field $\mathbb{K}$. For $\mathbb{K}$-modules $A,B$, we write the tensor product of $A$ and $B$ over $\mathbb{K}$ by $A\otimes B=A\otimes_{\mathbb{K}}B$ for convenience.

\subsection{Interaction homology}
An \emph{$n$-interaction simplicial complex} $(K,\{K_{i}\}_{1\leq i\leq n})$ consists of a simplicial complex $K$ equipped with a family of sub complexes $K_{1},\dots,K_{n}$ such that $K=\bigcup\limits_{i=1}^{n}K_{i}$. A \emph{morphism of $n$-interaction simplicial complexes} $(f,\{f_{i}\}_{1\leq i\leq j}):(K,\{K_{i}\}_{1\leq i\leq n})\to (L,\{L_{i}\}_{1\leq i\leq n})$ consists of  a family of simplicial maps $f_{i}:K_{i}\to L_{i}$ such that for any $1\leq i,j\leq n$, $f_{i}(\sigma_{i})=f_{j}(\sigma_{j})$ if and only if $\sigma_{i}=\sigma_{j}$, where $\sigma_{i}\in K_{i},\sigma_{j}\in K_{j}$. For convenience, we always refer to the interaction simplicial complex as $\{K_{i}\}_{1\leq i\leq n}$ and denote interaction morphism as $\{f_{i}\}_{1\leq i\leq n}:\{K_{i}\}_{1\leq i\leq n}\to\{L_{i}\}_{1\leq i\leq n}$.

Given an interaction simplicial complex $\{K_{i}\}_{1\leq i\leq n}$, we have a family of chain complexes $C_{\ast}(K_{i})$ for $i=1,2,\dots,n$. The tensor product $\bigotimes\limits_{i=1}^{n}C_{\ast}(K_{i})$ is naturally a chain complex with the differential given by
\begin{equation*}
  d(\sigma_{1}\otimes \sigma_{2}\otimes\cdots\otimes \sigma_{n})=\sum\limits_{i=1}^{n}(-1)^{p_{1}+\cdots+p_{i-1}}\sigma_{1}\otimes\cdots \otimes d\sigma_{i}\otimes\cdots\otimes \sigma_{n}.
\end{equation*}
Here, $\sigma_{i}\in K_{p_{i}}$ for $i=1,2,\dots,n$.
Let $D_{\ast}(\{K_{i}\}_{1\leq i\leq n})$ denote the sub $\mathbb{K}$-linear space of $\bigotimes\limits_{i=1}^{n}C_{\ast}(K_{i})$ with generators $\sigma_{1}\otimes \sigma_{2}\otimes\cdots\otimes \sigma_{n}$ satisfying $\bigcap\limits_{i=1}^{n} \sigma_{i}=\emptyset$. It can be verified that $D_{\ast}(\{K_{i}\}_{1\leq i\leq n})$ is the subcomplex of $\bigotimes\limits_{i=1}^{n}C_{\ast}(K_{i})$ with the inherited differential.
The \emph{interaction chain complex on $\{K_{i}\}_{1\leq i\leq n}$} is defined by
\begin{equation*}
  IC_{\ast}(\{K_{i}\}_{1\leq i\leq n}):=\left(\bigotimes\limits_{i=1}^{n}C_{\ast}(K_{i})\right)/D_{\ast}(\{K_{i}\}_{1\leq i\leq n}).
\end{equation*}
Thus the \emph{interaction homology of $\{K_{i}\}_{1\leq i\leq n}$} is given by
\begin{equation*}
  H_{p}(\{K_{i}\}_{1\leq i\leq n}):=H_{p}(IC_{\ast}(\{K_{i}\}_{1\leq i\leq n})),\quad p\geq 0.
\end{equation*}
The rank of $H_{p}(\{K_{i}\}_{1\leq i\leq n})$ is the \emph{interaction Betti number}, and we denote it by $\beta_{p}=\rank H_{p}(\{K_{i}\}_{1\leq i\leq n})$.

\begin{example}\label{example:star_point}
Consider the interaction complex $\{X_{1},X_{2}\}$ (see Fig. \ref{figure:examples}(a)), where
\begin{equation*}
  X_{1}=\{\{v_0\},\{v_1\},\{v_2\},\{v_3\},\{v_0,v_1\},\{v_0,v_2\},\{v_0,v_3\}\},\quad X_{2}=\{\{v_0\}\}.
\end{equation*}
The interaction chain complex $IC_{\ast}(\{X_{1},X_{2}\})$ is generated by the following elements
\begin{equation*}
  \{v_0\}\otimes \{v_0\}, \{v_0,v_{1}\}\otimes \{v_0\}, \{v_0,v_{2}\}\otimes \{v_0\}, \{v_0,v_{3}\}\otimes \{v_0\}.
\end{equation*}
The differential is given by $d_{0}=0$ and
\begin{equation*}
  d_{1}\left(
         \begin{array}{c}
           \{v_0,v_{1}\}\otimes \{v_0\} \\
           \{v_0,v_{2}\}\otimes \{v_0\} \\
           \{v_0,v_{3}\}\otimes \{v_0\} \\
         \end{array}
       \right)=\left(
                 \begin{array}{c}
                   -1 \\
                   -1 \\
                   -1 \\
                 \end{array}
               \right)\left(
                        \begin{array}{c}
                          \{v_0\}\otimes \{v_0\} \\
                        \end{array}
                      \right).
\end{equation*}
Thus the space of cycles of $IC_{\ast}(\{X_{1},X_{2}\})$ is generated by
\begin{equation*}
  \{v_0\}\otimes \{v_0\},\{v_0,v_{2}\}\otimes \{v_0\}-\{v_0,v_{1}\}\otimes \{v_0\},\{v_0,v_{3}\}\otimes \{v_0\}-\{v_0,v_{1}\}\otimes \{v_0\},
\end{equation*}
while the space of boundaries of $IC_{\ast}(\{X_{1},X_{2}\})$ is generated by $\{v_0\}\otimes \{v_0\}$. Hence, we have
\begin{equation*}
  H_{\ast}(IC_{\ast}(\{X_{1},X_{2}\}))=\mathrm{span}_{\mathbb{K}}\langle \{v_0,v_{2}\}\otimes \{v_0\}-\{v_0,v_{1}\}\otimes \{v_0\},\{v_0,v_{3}\}\otimes \{v_0\}-\{v_0,v_{1}\}\otimes \{v_0\}\rangle.
\end{equation*}
So the interaction homology of $\{X_{1},X_{2}\}$ is
\begin{equation*}
  H_{p}(\{X_{1},X_{2}\})\cong \left\{
                           \begin{array}{ll}
                             \mathbb{K}\oplus \mathbb{K}, & \hbox{$p=1$;} \\
                             0, & \hbox{otherwise.}
                           \end{array}
                         \right.
\end{equation*}
The corresponding Betti number for $\{X_{1},X_{2}\}$ is $\beta_{1}=2$ and $\beta_{p}=0$ for $p\neq 1$.
\end{example}

\begin{figure}[htbp]
  \centering
  \includegraphics[width=0.7\textwidth]{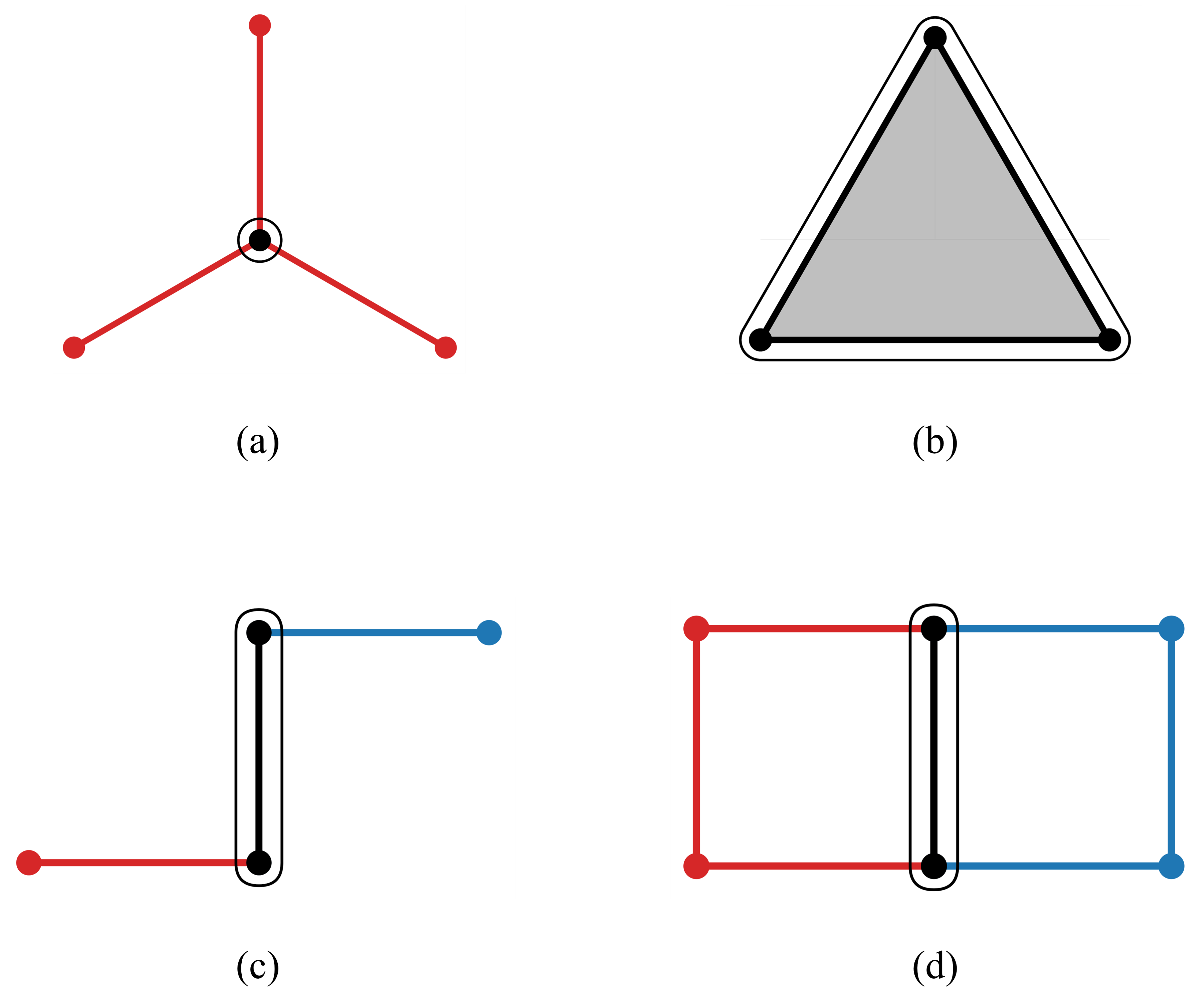}\\
  \caption{The visual representation of interaction complexes presented in examples. The black parts signify the intersection components, while the red and black areas delineate $X_{1}$, and similarly, the blue and black regions illustrate $X_{2}$. (a) The interaction complex showcased in Example \ref{example:star_point}; (b) The interaction complex $\{X_{1},X_{2}\}$ with $X_{1}=X_{2}=\Delta[2]$, as discussed in Example \ref{example:simplex}; (c) The interaction complex considered in Example \ref{example:zig_zag}; (d) The interaction complex that appeared in Example \ref{example:square_square}.}
	\label{figure:examples}
\end{figure}

\begin{example}\label{example:simplex}
Let $\Delta[n]$ be the simplicial complex of the standard $n$-simplex. Consider the interaction complex $\{X_{1},X_{2}\}$ such that $X_{1}=X_{2}=\Delta[n]$, see Fig. \ref{figure:examples}(b) for $n=2$. For the case $n=1$, the complex $\Delta[1]$ has the non-degenerate element $\{0\}$, $\{1\}$, and $\{0,1\}$. Then the interaction chain complex $IC_{\ast}(\{X_{1},X_{2}\})$ is generated by the following elements
\begin{equation*}
  \{0\}\otimes  \{0\},\{1\}\otimes  \{1\},\{0\}\otimes  \{0,1\},\{1\}\otimes  \{0,1\},\{0,1\}\otimes  \{0\},\{0,1\}\otimes  \{1\},\{0,1\}\otimes  \{0,1\}.
\end{equation*}
Besides, the differential of $IC_{\ast}(\{X_{1},X_{2}\})$ is given by $d_{0}=0$,
\begin{equation*}
  d_{1}\left(
         \begin{array}{c}
           \{0\}\otimes  \{0,1\}\\
           \{1\}\otimes  \{0,1\} \\
           \{0,1\}\otimes  \{0\} \\
           \{0,1\}\otimes  \{1\} \\
         \end{array}
       \right)=\left(
                 \begin{array}{cc}
                   -1 & 0 \\
                   0 & 1 \\
                   -1 & 0 \\
                   0 & 1 \\
                 \end{array}
               \right)\left(
                        \begin{array}{c}
                          \{0\}\otimes  \{0\} \\
                          \{1\}\otimes  \{1\} \\
                        \end{array}
                      \right).
\end{equation*}
and
\begin{equation*}
  d_{2}\left(
         \begin{array}{c}
           \{0,1\}\otimes  \{0,1\} \\
         \end{array}
       \right)=\left(
                 \begin{array}{cccc}
                   -1 & 1 & 1 & -1 \\
                 \end{array}
               \right)\left(
         \begin{array}{c}
           \{0\}\otimes  \{0,1\}\\
           \{1\}\otimes  \{0,1\} \\
           \{0,1\}\otimes  \{0\} \\
           \{0,1\}\otimes  \{1\} \\
         \end{array}
                      \right).
\end{equation*}
Let $B_{p}$ be the representation matrix of the differential $d_{p}$ with respect to the chosen basis. Recall that
\begin{equation*}
  \beta_{p}=\dim IC_{p} - \rank B_{p} -\rank B_{p+1}.
\end{equation*}
Thus we have $\beta_{1}=4-2-1=1$ and $\beta_{p}=0$ for $p\neq 1$. It follows that $H_{p}(\{X_{1},X_{2}\})=\left\{
                                                                                                            \begin{array}{ll}
                                                                                                              \mathbb{K}, & \hbox{$p=1$;} \\
                                                                                                              0, & \hbox{otherwise.}
                                                                                                            \end{array}
                                                                                                          \right.
$ It is quite different from the usual simplicial homology.

For the case $n=2$, one can verify that the interaction homology of the interaction complex $\{X_{1},X_{2}\}$ is $H_{p}(\{X_{1},X_{2}\})=\left\{
                                                                                                            \begin{array}{ll}
                                                                                                              \mathbb{K}, & \hbox{$p=2$;} \\
                                                                                                              0, & \hbox{otherwise.}
                                                                                                            \end{array}
                                                                                                          \right.
$
Indeed, the cycle
\begin{equation*}
  x=\{0\}\otimes \{0,1,2\}-\{0,1\}\otimes \{0,2\}+ \{0,2\}\otimes \{0,1\}+\{0,1,2\}\otimes \{0\}
\end{equation*}
gives the homology generator $[x]$. It is impossible for the element $x$ to be a boundary. If $x$ is a boundary of the form $\partial y$, we can write $y=a\{0,1\}\otimes \{0,1,2\}+b\{0,2\}\otimes \{0,1,2\}+c \{1,2\}\otimes \{0,1,2\} +y'$. Here, $y'$ is the sum of other terms. Note that $dy'$ does not contribute to the term of forms $\{0\}\otimes \{0,1,2\}$, $\{1\}\otimes \{0,1,2\}$ or $\{2\}\otimes \{0,1,2\}$. We have that
\begin{equation*}
  a(\{1\}-\{0\})\otimes \{0,1,2\}+b(\{2\}-\{0\})\otimes \{0,1,2\}+c(\{2\}-\{1\})\otimes \{0,1,2\}=\{0\}\otimes \{0,1,2\}.
\end{equation*}
It follows that $\left\{
                  \begin{array}{ll}
                    a+b=-1   \\
                    a-c=0   \\
                    b+c=0
                  \end{array}
                \right.$, which has no solution. Thus $x$ is not a boundary, and $[x]$ is indeed a homology generator.

In general, for any $n\geq 0$, we can obtain that $H_{p}(\{X_{1},X_{2}\})=\left\{
                                                                                                            \begin{array}{ll}
                                                                                                              \mathbb{K}, & \hbox{$p=n$;} \\
                                                                                                              0, & \hbox{otherwise.}
                                                                                                            \end{array}
                                                                                                          \right.$
\end{example}

Consider the case $n=2$ and $K_{1}=K_{2}=K$. The \emph{Wu characteristic} of $K$ is defined by $\omega(K):=\sum\limits_{\sigma\sim \tau}(-1)^{\dim\sigma+\dim\tau}$. Here, $\sigma\sim\tau$ means $\sigma\cap \tau\neq\emptyset$ for $\sigma,\tau\in K$. The Euler theorem asserts that for a simplicial complex $K$, one has
\begin{equation*}
  \mathcal{X}(K)=\sum\limits_{i\geq 0}(-1)^{i}\beta_{i}(K).
\end{equation*}
Here, $\mathcal{X}(K)$ is the Euler characteristic of $K$. Similarly, we have the theorem for Wu characteristic
\begin{equation*}
  \omega(K)=\sum\limits_{i\geq 0}(-1)^{i}\beta_{2,i}(K),
\end{equation*}
where $\beta_{2,i}(K)$ is the interaction Betti number for $\{K_{1},K_{2}\}$ with $K_{1}=K_{2}=K$. It is shown that the interaction homology decodes the Wu characteristic in topology.

\begin{example}
Let $\partial \Delta[2]$ be the boundary of a triangle. The elements of $\partial \Delta[2]$ can be listed as
\begin{equation*}
  \{0\},\{1\},\{2\},\{0,1\},\{0,2\},\{1,2\}.
\end{equation*}
On can prove that the interaction Betti number is $\beta_{2,i}(\partial \Delta[2])=\left\{
                                                                  \begin{array}{ll}
                                                                    1, & \hbox{$i=1,2$;} \\
                                                                    0, & \hbox{otherwise.}
                                                                  \end{array}
                                                                \right.
$ Let us count the interaction pairs. The 0-dimensional pairs are $(\{0\},\{0\})$, $(\{1\},\{1\})$, and $(\{2\},\{2\})$. The 1-dimensional pairs are listed as
\begin{equation*}
\begin{split}
   &  (\{0\},\{0,1\}),(\{0\},\{0,2\}),(\{1\},\{0,1\}),(\{1\},\{1,2\}),(\{2\},\{0,2\}),(\{2\},\{1,2\}), \\
    & (\{0,1\},\{0\}),(\{0,1\},\{1\}),(\{0,2\},\{0\}),(\{0,2\},\{2\}),(\{1,2\},\{1\}),(\{1,2\},\{2\}).
\end{split}
\end{equation*}
The 2-dimensional pairs are represented as
\begin{equation*}
\begin{split}
  &(\{0,1\},\{0,1\}),(\{0,1\},\{0,2\}),(\{0,1\},\{1,2\}),(\{0,2\},\{0,1\}),(\{0,2\},\{0,2\}),\\
  &(\{0,2\},\{1,2\}),(\{1,2\},\{0,1\}),(\{1,2\},\{0,2\}),(\{1,2\},\{1,2\}).
\end{split}
\end{equation*}
There is no interaction pairs of dimension $\geq 3$. It follows that the Wu characteristic is
\begin{equation*}
  \omega(\partial \Delta[2])=\sum\limits_{\sigma\sim \tau}(-1)^{\dim\sigma+\dim\tau}= 3 - 12 + 9 = 0.
\end{equation*}
This is the same as the alternating sum of the interaction Betti numbers
\begin{equation*}
  \omega(\partial \Delta[2]) =\sum\limits_{i\geq 0}(-1)^{i}\beta_{2,i}(\partial \Delta[2]) = 0 -  1 + 1=0.
\end{equation*}

For the example of $\Delta[2]$, one can verify the presence of three 0-dimensional interaction pairs, twelve 1-dimensional interaction pairs, fifteen 2-dimensional interaction pairs, six 3-dimensional interaction pairs, and one 4-dimensional interaction pair. The corresponding Wu characteristic is
\begin{equation*}
  \omega(  \Delta[2]) =\sum\limits_{\sigma\sim \tau}(-1)^{\dim\sigma+\dim\tau}= 3 - 12 + 15 -6 + 1 = 1.
\end{equation*}
This aligns with the calculations in Example \ref{example:simplex}, showing that
\begin{equation*}
  \omega(  \Delta[2]) =\sum\limits_{i\geq 0}(-1)^{i}\beta_{2,i}(\Delta[2]) = 1.
\end{equation*}

\end{example}

Consider the category $\mathbf{IC_{n}}$ of interaction simplicial complexes. The objects in $\mathbf{IC_{n}}$ are the $n$-interaction simplicial complexes, and the morphisms are the morphisms of $n$-interaction simplicial complexes.

For each morphism $\{f_{i}\}_{1\leq i\leq n}:\{K_{i}\}_{1\leq i\leq n}\to \{L_{i}\}_{1\leq i\leq n}$, we have a morphism of chain complexes
\begin{equation*}
  \bigotimes\limits_{i=1}^{n}C_{\ast}(f_{i}): \bigotimes\limits_{i=1}^{n}C_{\ast}(K_{i})\to \bigotimes\limits_{i=1}^{n}C_{\ast}(L_{i}).
\end{equation*}
Composite with the quotient map, we obtain a morphism of chain complexes
\begin{equation*}
   \alpha:\bigotimes\limits_{i=1}^{n}C_{\ast}(K_{i})\to  IC_{\ast}(\{L_{i}\}_{1\leq i\leq n}).
\end{equation*}
Note that $D_{\ast}(\{K_{i}\}_{1\leq i\leq n})\subseteq \ker \alpha$. The morphism of chain complex $\alpha$ induces a morphism of chain complexes
\begin{equation*}
   \theta_{f}:IC_{\ast}(\{K_{i}\}_{1\leq i\leq n})\to  IC_{\ast}(\{L_{i}\}_{1\leq i\leq n})
\end{equation*}
given by $\theta_{f}(\sum\limits_{\sigma_{1},\dots,\sigma_{n}}a_{\sigma_{1},\dots,\sigma_{n}}\sigma_{1}\otimes \cdots\otimes \sigma_{n})=\sum\limits_{\sigma_{1},\dots,\sigma_{n}}a_{\sigma_{1},\dots,\sigma_{n}}f_{1}(\sigma_{1})\otimes \cdots\otimes f_{n}(\sigma_{n})$.
The morphism $\theta_{f}$ induces a map of homology
\begin{equation*}
  H_{\ast}(\theta_{f}):H_{\ast}(\{K_{i}\}_{1\leq i\leq n})\to  H_{\ast}(\{L_{i}\}_{1\leq i\leq n}),[z]\mapsto [\theta_{f}(z)].
\end{equation*}
The following result shows that the interaction homology from the category of interaction complexes to the category of vector spaces is a functor.
\begin{proposition}\label{proposition:functor}
The interaction homology $H_{\ast}:\mathbf{IC_{n}}\to \mathbf{Vec}_{\mathbb{K}}$ is functorial.
\end{proposition}
\begin{proof}
Let $\{f_{i}\}_{1\leq i\leq n}:\{K_{i}\}_{1\leq i\leq n}\to \{L_{i}\}_{1\leq i\leq n}$ and $\{g_{i}\}_{1\leq i\leq n}:\{L_{i}\}_{1\leq i\leq n}\to \{M_{i}\}_{1\leq i\leq n}$ be morphisms of interaction simplicial complexes. Then for any cycle $z=\sum\limits_{\sigma_{1},\dots,\sigma_{n}}a_{\sigma_{1},\dots,\sigma_{n}}\sigma_{1}\otimes \cdots\otimes \sigma_{n}\in IC_{\ast}(\{K_{i}\}_{1\leq i\leq n})$, one has
\begin{equation*}
\begin{split}
  H_{\ast}(\theta_{g})\circ H_{\ast}(\theta_{f})([z])= & H_{\ast}(\theta_{g})\left([\sum\limits_{\sigma_{1},\dots,\sigma_{n}}a_{\sigma_{1},\dots,\sigma_{n}}f_{1}(\sigma_{1})\otimes \cdots\otimes f_{n}(\sigma_{n})]\right) \\
    =& [\sum\limits_{\sigma_{1},\dots,\sigma_{n}}a_{\sigma_{1},\dots,\sigma_{n}}g_{1}f_{1}(\sigma_{1})\otimes \cdots\otimes g_{n}f_{n}(\sigma_{n})]\\
    =&  H_{\ast}(\theta_{gf})([z]).
\end{split}
\end{equation*}
By verifying the definition, we obtain that $H_{\ast}$ is functorial.
\end{proof}

\subsection{Interaction Laplacian}\label{section:interaction_laplacian}

On the chain complex of simplicial complexes, there exists a combinatorial Laplacian, and its kernel space is commonly referred to as the harmonic space, which is isomorphic to the simplicial homology of the simplicial complex. In addition to this, researchers have explored Laplacians on path complexes, hypergraphs, and hyperdigraphs. These Laplacians serve to reflect the topological and geometric characteristics of various objects, holding significant potential in data analysis applications such as molecular structure analysis and material structure analysis. In this section, we will explore the Laplacian operator on interaction complexes from now on. The ground field $\mathbb{K}$ is taken to be the real number field $\mathbb{R}$.

Let $\{K_{i}\}_{1\leq i\leq n}$ be an interaction simplicial complex. One can obtain an interaction chain complex $IC_{\ast}(\{K_{i}\}_{1\leq i\leq n})$. We endow $\bigotimes\limits_{i=1}^{n}C_{\ast}(K_{i})$ with the inner product
\begin{equation*}
  \langle \sigma_{1}\otimes\cdots\otimes \sigma_{n},\tau_{1}\otimes\cdots\otimes \tau_{n}\rangle = \prod\limits_{i=1}^{n}\langle \sigma_{i},\tau_{i}\rangle,
\end{equation*}
where $\sigma_{1}\otimes\cdots\otimes \sigma_{n},\tau_{1}\otimes\cdots\otimes \tau_{n}\in \bigotimes\limits_{i=1}^{n}C_{\ast}(K_{i})$ and $\langle \sigma,\tau\rangle=\left\{
                                    \begin{array}{ll}
                                      1, & \hbox{$\sigma=\tau$;} \\
                                      0, & \hbox{otherwise.}
                                    \end{array}
                                  \right.
$
It induces the quotient inner product structure on $IC_{\ast}(\{K_{i}\}_{1\leq i\leq n})$, also denoted by $\langle\cdot,\cdot\rangle$. Note that $IC_{p}(\{K_{i}\}_{1\leq i\leq n})$ and $IC_{q}(\{K_{i}\}_{1\leq i\leq n})$ are orthogonal for $p\neq q$.
Let $d_{p}:IC_{p}(\{K_{i}\}_{1\leq i\leq n})\to IC_{p-1}(\{K_{i}\}_{1\leq i\leq n})$ be the differential at degree $p$. We have the adjoint operator $(d_{p})^{\ast}$, which is given by
\begin{equation*}
  \langle d_{p}x,y\rangle=\langle x,(d_{p})^{\ast}y\rangle
\end{equation*}
for all $x\in IC_{p}(\{K_{i}\}_{1\leq i\leq n}),y\in IC_{p-1}(\{K_{i}\}_{1\leq i\leq n})$. The \emph{$p$-th interaction Laplacian} $\Delta_{p}:IC_{p}(\{K_{i}\}_{1\leq i\leq n})\to IC_{p}(\{K_{i}\}_{1\leq i\leq n})$ is defined by
\begin{equation*}
  \Delta_{p}:=(d_{p})^{\ast}\circ d_{p}+d_{p+1}\circ (d_{p+1})^{\ast},\quad p\geq 0.
\end{equation*}
In particular, $\Delta_{0}=d_{1}\circ (d_{1})^{\ast}$. Let $B_{p}$ be the representation matrix of $B_{p}$ with respect to the standard orthogonal basis. Then the representation matrix of $\Delta_{p}$ is given by
\begin{equation*}
  L_{p} = B_{p}B_{p}^{T}+B_{p+1}^{T}B_{p+1}.
\end{equation*}
The zero eigenvalues of the Laplacian reflect information about its harmonic components, while the non-zero eigenvalues indicate information about its non-harmonic components. Among these, the smallest positive eigenvalue of $L_{p}$, known as the spectral gap, is the most commonly used non-harmonic feature.

\begin{proposition}
The Laplacian $\Delta_{p}$ on $IC_{p}(\{K_{i}\}_{1\leq i\leq n})$ is self-adjoint and non-negative definite.
\end{proposition}
The non-negative definiteness ensures that the eigenvalues of the operator $\Delta_{p}$ are non-negative. Moreover, we have the algebraic Hodge decomposition on $IC_{p}(\{K_{i}\}_{1\leq i\leq n})$.

\begin{proposition}
$IC_{p}(\{K_{i}\}_{1\leq i\leq n})=\ker \Delta_{p}\oplus \im d_{p+1}\oplus \im (d_{p})^{\ast}$. Here, $\ker \Delta_{p}$ is isomorphic to the interaction homology $H_{p}(\{K_{i}\}_{1\leq i\leq n})$.
\end{proposition}
The proofs of the above results are standard and can be found in detail in \cite{liu2023algebraic}.

\begin{example}\label{example:zig_zag}
Consider the interaction complex $\{K_{1},K_{2}\}$ (see Fig. \ref{figure:examples} (c)), where
\begin{equation*}
\begin{split}
  K_{1}&=\{\{v_{0}\},\{v_{1}\},\{v_{2}\},\{v_{0},v_{1}\},\{v_{1},v_{2}\}\},\\
  K_{2}&=\{\{v_{1}\},\{v_{2}\},\{v_{3}\},\{v_{1},v_{2}\},\{v_{2},v_{3}\}\}.
\end{split}
\end{equation*}
The interaction chain complex $IC_{\ast}(\{K_{1},K_{2}\})$ of $\{K_{1},K_{2}\}$ is generated by
\begin{equation*}
\begin{split}
    & \{v_{1}\}\otimes \{v_{1}\}, \{v_{2}\}\otimes \{v_{2}\}, \\
    & \{v_{1}\}\otimes \{v_{1},v_{2}\}, \{v_{2}\}\otimes \{v_{1},v_{2}\}, \{v_{2}\}\otimes \{v_{2},v_{3}\},\{v_{0},v_{1}\}\otimes \{v_{1}\},\{v_{1},v_{2}\}\otimes \{v_{1}\},\{v_{1},v_{2}\}\otimes \{v_{2}\},\\
    & \{v_{0},v_{1}\}\otimes \{v_{1},v_{2}\}, \{v_{1},v_{2}\}\otimes \{v_{1},v_{2}\}, \{v_{1},v_{2}\}\otimes \{v_{2},v_{3}\}.
\end{split}
\end{equation*}
The corresponding differential is given by $d_{0}=0$,
\begin{equation*}
  d_{1}\left(
         \begin{array}{c}
           \{v_{1}\}\otimes \{v_{1},v_{2}\} \\
           \{v_{2}\}\otimes \{v_{1},v_{2}\} \\
           \{v_{2}\}\otimes \{v_{2},v_{3}\} \\
           \{v_{0},v_{1}\}\otimes \{v_{1}\} \\
           \{v_{1},v_{2}\}\otimes \{v_{1}\} \\
           \{v_{1},v_{2}\}\otimes \{v_{2}\} \\
         \end{array}
       \right)=\left(
                 \begin{array}{cc}
                   -1 & 0 \\
                   0 & 1 \\
                   0 & -1 \\
                   1 & 0 \\
                   -1 & 0 \\
                   0 & 1 \\
                 \end{array}
               \right)\left(
                        \begin{array}{c}
                          \{v_{1}\}\otimes \{v_{1}\} \\
                          \{v_{2}\}\otimes \{v_{2}\} \\
                        \end{array}
                      \right),
\end{equation*}
and
\begin{equation*}
  d_{2}\left(
         \begin{array}{c}
           \{v_{0},v_{1}\}\otimes \{v_{1},v_{2}\} \\
           \{v_{1},v_{2}\}\otimes \{v_{1},v_{2}\} \\
           \{v_{1},v_{2}\}\otimes \{v_{2},v_{3}\} \\
         \end{array}
       \right)=\left(
                 \begin{array}{cccccc}
                   1 & 0 & 0 & 1 & 0 & 0 \\
                   -1 & 1 & 0 & 0 & 1 & -1 \\
                   0 & 0 & 1 & 0 & 0 & 1 \\
                 \end{array}
               \right)
\left(
         \begin{array}{c}
           \{v_{1}\}\otimes \{v_{1},v_{2}\} \\
           \{v_{2}\}\otimes \{v_{1},v_{2}\} \\
           \{v_{2}\}\otimes \{v_{2},v_{3}\} \\
           \{v_{0},v_{1}\}\otimes \{v_{1}\} \\
           \{v_{1},v_{2}\}\otimes \{v_{1}\} \\
           \{v_{1},v_{2}\}\otimes \{v_{2}\} \\
         \end{array}
       \right).
\end{equation*}
We denote the representation matrix of $d_{p}$ by $B_{p}$. The Laplacian matrices are
\begin{equation*}
  L_{0}=B_{1}^{T}B_{1}=\left(
                         \begin{array}{cc}
                           3 & 0 \\
                           0 & 3 \\
                         \end{array}
                       \right),
\end{equation*}
\begin{equation*}
  L_{1}=B_{1}B_{1}^{T}+B_{2}^{T}B_{2}=\left(
                                                \begin{array}{cccccc}
                                                  3 & -1 & 0 & 0 & 0 & 1 \\
-1 & 2 & -1 & 0 & 1 & 0 \\
0 & -1 & 2 & 0 & 0 & 0 \\
0 & 0 & 0 & 2 & -1 & 0 \\
0 & 1 & 0 & -1 & 2 & -1 \\
1 & 0 & 0 & 0 & -1 & 3 \\
                                                \end{array}
                                              \right),
\end{equation*}
and
\begin{equation*}
  L_{2}=B_{2}B_{2}^{T}=\left(
                         \begin{array}{ccc}
2 & -1 & 0 \\
-1 & 4 & -1 \\
0 & -1 & 2 \\
                         \end{array}
                       \right).
\end{equation*}
The spectra of Laplacians are $\mathbf{Spec}(L_{0})=\{3, 3\}$, $\mathbf{Spec}(L_{1})=\{0, 3-\sqrt{3}, 2, 3, 3, 3+\sqrt{3}\}$, and $\mathbf{Spec}(L_{2})=\{3-\sqrt{3}, 2, 3+\sqrt{3}\}$.
\end{example}

\begin{example}\label{example:square_square}
Consider the interaction complex $\{X_{1},X_{2}\}$ (see Fig. \ref{figure:examples} (d)), where
\begin{equation*}
  X_{1}=\{\{0\},\{1\},\{2\},\{3\},\{0,1\},\{0,2\},\{1,3\},\{2,3\}\}
\end{equation*}
and
\begin{equation*}
  X_{2}=\{\{2\},\{3\},\{4\},\{5\},\{2,3\},\{2,4\},\{3,5\},\{4,5\}\}.
\end{equation*}
Then the interaction chain complex $IC_{\ast}(\{X_{1},X_{2}\})$ has the generators
\begin{equation*}
  \begin{split}
      & \text{0-dim: } \{2\}\otimes \{2\}, \{3\}\otimes \{3\}, \\
      & \text{1-dim: }\{2\}\otimes \{2,3\},\{2\}\otimes \{2,4\},\{3\}\otimes \{2,3\},\{3\}\otimes \{3,5\},\{0,2\}\otimes \{2\}, \{2,3\}\otimes \{2\},\\
      &  \{2,3\}\otimes \{3\}, \{3,5\}\otimes \{3\}, \\
      &\text{2-dim: }\{0,2\}\otimes \{2,3\},\{0,2\}\otimes \{2,4\},\{1,3\}\otimes \{2,3\}, \{1,3\}\otimes \{3,5\},\{2,3\}\otimes \{2,3\},\\
      & \{2,3\}\otimes \{2,4\},\{2,3\}\otimes \{3,5\}.
  \end{split}
\end{equation*}
Choose the interaction simplices as orthogonal basis. Then the interaction homology is
\begin{equation*}
  H_{p}(\{X_{1},X_{2}\})=\left\{
                           \begin{array}{ll}
                             \mathbb{K}, & \hbox{$p=2$;} \\
                             0, & \hbox{otherwise.}
                           \end{array}
                         \right.
\end{equation*}
Moreover, the calculation of Laplacians are shown in Table \ref{table:squre_squre}.
\begin{table}[htb!]
  \centering
  \caption{Illustration of interaction Laplacians in Example \ref{example:square_square}.}\label{table:squre_squre}
  \begin{footnotesize}
  \begin{tabular}{c|c|c|c}
    \hline
    $p$ & $p=0$ & $p=1$& $p=2$  \\
    \hline
    $L_{p}$ & $
  \left(
    \begin{array}{cc}
4 & 0 \\
0 & 4 \\
    \end{array}
  \right)$
     & $
  \left(
    \begin{array}{cccccccc}
3 & 1 & 0 & 0 & 0 & 0 & 0 & 0 \\
1 & 3 & -1 & 0 & 0 & 0 & -1 & 0 \\
0 & -1 & 3 & -1 & 0 & -1 & 0 & 0 \\
0 & 0 & -1 & 3 & 0 & 0 & 0 & 0 \\
0 & 0 & 0 & 0 & 3 & 0 & 1 & 0 \\
0 & 0 & -1 & 0 & 0 & 3 & -1 & -1 \\
0 & -1 & 0 & 0 & 1 & -1 & 3 & 0 \\
0 & 0 & 0 & 0 & 0 & -1 & 0 & 3 \\
    \end{array}
  \right)$
     & $
  \left(
    \begin{array}{ccccccc}
2 & -1 & 0 & 1 & 0 & 0 & 0 \\
-1 & 2 & 0 & 0 & 1 & 0 & 0 \\
0 & 0 & 2 & 1 & 0 & -1 & 0 \\
1 & 0 & 1 & 4 & -1 & 0 & -1 \\
0 & 1 & 0 & -1 & 2 & 0 & 0 \\
0 & 0 & -1 & 0 & 0 & 2 & 1 \\
0 & 0 & 0 & -1 & 0 & 1 & 2 \\
    \end{array}
  \right)$\\
         \hline
    $\beta_{p}$ & 0 & 0 & 1 \\
      \hline
    $\mathbf{Spec}(L_{p})$ &\{4,4\}& $\{2-\sqrt{2},2,2,4-\sqrt{2},2+\sqrt{2},4,4,4+\sqrt{2}\}$ & $\{0,2-\sqrt{2},2,2,4-\sqrt{2},2+\sqrt{2},4+\sqrt{2}\}$\\
    \hline
  \end{tabular}
  \end{footnotesize}
  \end{table}

\end{example}

\section{Persistence on interaction complexes}\label{section:PIH}

\subsection{Persistent interaction homology}
While persistent homology captures the topological features and geometric shapes of spaces, persistent interaction homology provides a topological description of the interactions between spaces. By analyzing the interactions between spaces, we can gain insights into how the topology of the spaces affects their behavior and dynamics. This can have applications in a variety of fields, including biology, physics, and social sciences. In this section, the ground field is assumed to be the field $\mathbb{K}$. If we consider the Laplacian, the ground field $\mathbb{K}$ is taken to be the real number field $\mathbb{R}$.

Let $(X,\leq)$ be a poset with the partial order $\leq$. Then we can regard $(X,\leq)$ as a category with objects given by the elements in $X$ and morphisms of the form $a\to b$ for any $a\leq b$. A \emph{persistence object} on a category $\mathfrak{C}$ is a functor $\mathcal{F}:(X,\leq)\to \mathfrak{C}$.

A \emph{persistence $n$-interaction (simplicial) complex} is a functor $\mathcal{F}=\{\mathcal{F}_{i}\}_{1\leq i\leq n}:(X,\leq)\to \mathbf{IC_{n}}$. By Proposition \ref{proposition:functor}, we have a persistence module
\begin{equation*}
  H_{\ast}\mathcal{F}:(X,\leq)\to \mathbf{Vec}_{\mathbb{K}}.
\end{equation*}
For any elements $a\leq b$ in $X$, the \emph{$(a,b)$-persistent interaction homology} is defined by
\begin{equation*}
  H^{a,b}_{\ast}(\mathcal{F};\mathbb{K}):=\im(H_{\ast}(\mathcal{F}(a);\mathbb{K})\to H_{\ast}(\mathcal{F}(a);\mathbb{K})).
\end{equation*}
The $(a,b)$-persistent interaction Betti number is $\beta_{p}^{a,b}=\dim H^{a,b}_{p}(\mathcal{F};\mathbb{K})$. Just like the traditional persistent homology, the corresponding persistent Betti numbers of persistent interaction homology can also be represented using persistence diagrams and barcodes. The persistence diagrams and barcodes provide a concise and intuitive way to capture the topological information that is persistent across a range of distance or interaction scales.

Typically, the poset $(X,\leq)$ is taken to be the ordered set of integers $(\mathbb{Z},\leq)$ with the usual order.
Classical theorems in persistent homology theory, such as the decomposition theorem of persistence modules and the interval theorem for barcodes, can also be applied to persistent interaction homology. These theorems provide fundamental mathematical tools for analyzing the structure and properties of persistent homology, and can be used to gain deeper insights into the topological features of interacting systems.

\subsection{Stability for persistent interaction homology}

In \cite{chazal2009proximity}, the authors introduced the interleaving distance to describe the algebraic stability theorem for persistence diagrams. Interleaving distance has been shown to be a generalization of the bottleneck distance \cite{bubenik2014categorification} and has become a fundamental tool for analyzing the stability of persistence objects \cite{bauer2020persistence,edelsbrunner2015persistent,liu2023algebraic}. In this section, we will provide a brief overview of the definition of interleaving distance and present the stability theorem for persistent interaction homology.

Let $\mathfrak{C}$ be a category, and let $\mathfrak{C}^{\mathbb{R}}$ be the category of functors from $(\mathbb{R},\leq)$ to $\mathfrak{C}$. Then for any functor $\mathcal{F}:(\mathbb{R},\leq)\to \mathfrak{C}$, we have a functor $\Sigma^{\varepsilon}\mathcal{F}:(\mathbb{R},\leq)\to \mathfrak{C}$ given by $(\Sigma^{\varepsilon}\mathcal{F})(a)=\mathcal{F}(a+\varepsilon)$.
Obviously, $\Sigma^{\varepsilon}|_{\mathcal{F}}:\mathcal{F}\Rightarrow\Sigma^{\varepsilon}\mathcal{F}$ is a natural transformation.

\begin{definition}
Let $\mathcal{F},\mathcal{G}:(\mathbb{R},\leq)\to \mathfrak{C}$ be two persistence objects. An \emph{$\varepsilon$-interleaving between $\mathcal{F}$ and $\mathcal{G}$} consists of two natural transformations $\phi:\mathcal{F}\Rightarrow \Sigma^{\varepsilon}\mathcal{G}$ and $\psi:\mathcal{G}\Rightarrow \Sigma^{\varepsilon}\mathcal{F}$ such that $(\Sigma^{\varepsilon}\psi)\phi=\Sigma^{2\varepsilon}|_{\mathcal{F}}$ and $(\Sigma^{\varepsilon}\phi)\psi=\Sigma^{2\varepsilon}|_{\mathcal{G}}$. We say that $\mathcal{F}$ and $\mathcal{G}$ are \emph{$\varepsilon$-interleaved}.
\begin{equation*}
  \xymatrix@=0.6cm{
  &\Sigma^{\varepsilon}\mathcal{G}\ar@{->}[rd]^{\Sigma^{\varepsilon}\psi}&\\
  \mathcal{F}\ar@{->}[ru]^{\phi}\ar@{->}[rr]^{\Sigma^{2\varepsilon}|_{\mathcal{F}}}&&\Sigma^{2\varepsilon}\mathcal{F}
  }\qquad \qquad
  \xymatrix@=0.6cm{
  &\Sigma^{\varepsilon}\mathcal{F}\ar@{->}[rd]^{\Sigma^{\varepsilon}\phi}&\\
  \mathcal{G}\ar@{->}[ru]^{\psi}\ar@{->}[rr]^{\Sigma^{2\varepsilon}|_{\mathcal{G}}}&&\Sigma^{2\varepsilon}\mathcal{G}
  }
\end{equation*}
\end{definition}
The \emph{interleaving distance} between $\mathcal{F}$ and $\mathcal{G}$ is defined by
\begin{equation*}
  d_{I}(\mathcal{F},\mathcal{G}):=\inf\{\varepsilon\geq 0|\text{$\mathcal{F}$ and $\mathcal{G}$ are $\varepsilon$-interleaved}\}.
\end{equation*}
\begin{theorem}\label{theorem:stability}
Let $\mathcal{F},\mathcal{G}:(\mathbb{R},\leq)\to \mathbf{IC_{n}}$ be two persistence interaction complexes. We have
\begin{equation*}
  d_{I}(H_{\ast}\mathcal{F},H_{\ast}\mathcal{G})\leq d_{I}(\mathcal{F},\mathcal{G}).
\end{equation*}
\end{theorem}
\begin{proof}
The proof is completed by Proposition \ref{proposition:functor} and \cite[Proposition 3.6]{bubenik2014categorification}.
\end{proof}

Consider an interaction complex $\{K_{i}\}_{1\leq i\leq n}$ equipped with a non-decreasing real-valued function
\begin{equation*}
  f=(f_{1},\dots,f_{n}):\{K_{i}\}_{1\leq i\leq n}\to \mathbb{R},
\end{equation*}
that is, $f_{i}(\sigma)\leq f_{i}(\tau)$ whenever $\sigma$ is a face of $\tau$ in $K_{i}$ for any $i=1,\dots,n$. For each real number $a$, we have an interaction complex
\begin{equation*}
  \mathcal{F}^{f}(a)=f^{-1}((-\infty,a])=\{(\sigma_{1},\dots,\sigma_{n})\in \{K_{i}\}_{1\leq i\leq n}|f_{i}(\sigma_{i})\leq a,\forall 1\leq i\leq n\}.
\end{equation*}
This construction gives us a persistence interaction complex $\mathcal{F}^{f}:(\mathbb{R},\leq)\to \mathbf{IC_{n}}$.

Now, given two non-decreasing real-valued functions $f,g:\{K_{i}\}_{1\leq i\leq n}\to \mathbb{R}$, we define the \emph{distance} between $f$ and $g$ on $\{K_{i}\}_{1\leq i\leq n}$ as
\begin{equation*}
  \|f-g\|_{\infty}=\sup\limits_{x\in \{K_{i}\}_{1\leq i\leq n}} \|f(x)-g(x)\|.
\end{equation*}
Here, $\|f(\sigma_{1},\dots,\sigma_{n})-g(\sigma_{1},\dots,\sigma_{n})\|=\max\limits_{1\leq i\leq n} |f_{i}(\sigma_{i})-g_{i}(\sigma_{i})|$.
\begin{theorem}
Let $\{K_{i}\}_{1\leq i\leq n}$ be an interaction complex equipped with two non-decreasing real-valued functions $f,g:\{K_{i}\}_{1\leq i\leq n}\to \mathbb{R}$. We have
\begin{equation*}
  d_{I}(H_{\ast}\mathcal{F}^{f},H_{\ast}\mathcal{F}^{g})\leq \|f-g\|_{\infty}.
\end{equation*}
\end{theorem}
\begin{proof}
Let $\varepsilon=\|f-g\|_{\infty}$. Then there are inclusion morphisms $\mathcal{F}^{f}(a)\hookrightarrow \mathcal{F}^{g}(a+\varepsilon)$ and $\mathcal{F}^{g}(a)\hookrightarrow \mathcal{F}^{f}(a+\varepsilon)$ for any real number $a\in \mathbb{R}$. The inclusions induce the natural transformations $\phi:\mathcal{F}^{f}(a)\Rightarrow \Sigma^{\varepsilon}\mathcal{F}^{g}(a)$ and $\psi:\mathcal{F}^{g}(a)\Rightarrow \Sigma^{\varepsilon}\mathcal{F}^{f}(a)$. It follows that
\begin{equation*}
  (\Sigma^{\varepsilon}\psi)\phi: \mathcal{F}^{f}\Rightarrow \mathcal{F}^{f}
\end{equation*}
is an inclusion. So we have $(\Sigma^{\varepsilon}\psi)\phi = \Sigma^{2\varepsilon}|_{\mathcal{F}^{f}}$. Similarly, one has $(\Sigma^{\varepsilon}\phi)\psi = \Sigma^{2\varepsilon}|_{\mathcal{F}^{g}}$. Hence, the persistence interaction complex $\mathcal{F}^{f}$ and $\mathcal{F}^{g}$ are $\varepsilon$-interleaved. By Theorem \ref{theorem:stability}, one has
\begin{equation*}
  d_{I}(H_{\ast}\mathcal{F}^{f},H_{\ast}\mathcal{F}^{g})\leq d_{I}(\mathcal{F}^{f},\mathcal{F}^{g})\leq \varepsilon=\|f-g\|_{\infty}.
\end{equation*}
The desired result follows.
\end{proof}

\subsection{Interaction Vietoris-Rips complexes}
The Vietoris-Rips complex is a commonly used geometric object that can be constructed from a set of points in a metric space. It is a simplicial complex that captures the topological features of the point cloud data, such as its connected components, loops, voids, and higher-dimensional holes. To study the interactions between a family of sets of points, one can construct the corresponding interaction Vietoris-Rips complex. This complex is formed by taking the Vietoris-Rips complexes of each set of points separately and then forming their intersection. The interaction Vietoris-Rips complex is a powerful tool in the study of complex systems where multiple interacting components are present.

Let $X_{1},X_{2},\dots,X_{n}$ be a family of sets of points in Euclidean space. Here, the point sets $X_{1},X_{2},\dots,X_{n}$ do not require to be disjoint. We can obtain a family of Vietoris-Rips complexes $\mathcal{R}_{1}(\varepsilon),\mathcal{R}_{2}(\varepsilon),\dots,\mathcal{R}_{n}(\varepsilon)$. Thus one has an interaction Vietoris-Rips complex $\mathcal{IR}(\varepsilon)=\{\mathcal{R}_{i}(\varepsilon)\}_{1\leq i\leq n}$. Moreover, it gives a persistence interaction complex
\begin{equation*}
  \mathcal{IR}:(\mathbb{R},\leq)\to \mathbf{IC_{n}},\quad \varepsilon\mapsto\{\mathcal{R}_{1}(\varepsilon),\dots,\mathcal{R}_{n}(\varepsilon)\}.
\end{equation*}
Indeed, for real numbers $\varepsilon_{0}\leq \varepsilon_{1}\leq \cdots\leq \varepsilon_{n}$, we have a filtration of interaction complexes
\begin{equation*}
  \mathcal{IR}(\varepsilon_{0})\hookrightarrow \mathcal{IR}(\varepsilon_{1})\hookrightarrow\cdots \hookrightarrow\mathcal{IR}(\varepsilon_{n}).
\end{equation*}
It induces a filtration of interaction homology
\begin{equation*}
  H_{\ast}(\mathcal{IR}(\varepsilon_{0}))\rightarrow H_{\ast}(\mathcal{IR}(\varepsilon_{1}))\rightarrow\cdots \rightarrow H_{\ast}(\mathcal{IR}(\varepsilon_{n})).
\end{equation*}
For $\varepsilon\leq \varepsilon'$, the $(\varepsilon,\varepsilon')$-persistent interaction homology for $X_{1},X_{2},\dots,X_{n}$ is given by
\begin{equation*}
  H^{\varepsilon,\varepsilon'}_{p}=\im (H_{p}(\{\mathcal{R}_{i}(\varepsilon)\}_{1\leq i\leq n})\to H_{p}(\{\mathcal{R}_{i}(\varepsilon')\}_{1\leq i\leq n})
\end{equation*}
for $p\geq 0$.

\begin{example}\label{example:interaction_rips}
Consider two point sets embedded in Euclidean space
\begin{equation*}
  X_{1}=\{(0,0),(1,0),(1,1)\},\quad X_{2}=\{(1,0),(1,1),(2,1)\}.
\end{equation*}
As the filtration parameter $\varepsilon$ grows from $0$ to $+\infty$. One has the interaction complex $\{\mathcal{R}_{1}(\varepsilon),\mathcal{R}_{2}(\varepsilon)\}$ changes at the parameters $\varepsilon=0,1,\sqrt{2}$. The filtration of interaction complexes is shown in Figure \ref{figure:interaction_filtration}(a). Point set $X_{1}$ consists of red and black points, while point set $X_{2}$ is composed of blue and black points. The black points represent the intersection of point sets $X_{1}$ and $X_{2}$.

The persistent interaction homology of ${X_{1},X_{2}}$ provides multiscale topological features of the interaction between the point sets $X_{1}$ and $X_{2}$. Figure \ref{figure:interaction_filtration}(d) illustrates the barcode for the filtration of interaction complexes derived from the interaction of $X_{1}$ and $X_{2}$.
On the other hand, we also examine the Vietoris-Rips complexes on the union of $X_{1}$ and $X_{2}$, as well as the Vietoris-Rips complexes on the intersection of $X_{1}$ and $X_{2}$. The corresponding barcodes are depicted in Figure \ref{figure:interaction_filtration}(b) and \ref{figure:interaction_filtration}(c). In this particular example, we observe that the interaction Vietoris-Rips complexes can capture changes in the filtration parameter at 1 and $\sqrt{2}$, whereas the Vietoris-Rips complexes on $X_{1}\cup X_{2}$ and $X_{1}\cap X_{2}$ only reveal variations at the value of 1. Certainly, this example does not conclusively demonstrate that the interaction Betti numbers contain more information than the Betti numbers of ordinary simplicial homology. However, it does highlight a key point: interaction Betti numbers can capture distinct topological features when multiple point sets are intertwined. Even in cases where the complexes may be contractible, interaction Betti numbers can still exist. In fact, when it comes to one-dimensional features like persistent Betti numbers, it is challenging to definitively determine superiority.
\begin{figure}[htb!]
  \centering
  \includegraphics[width=0.8\textwidth]{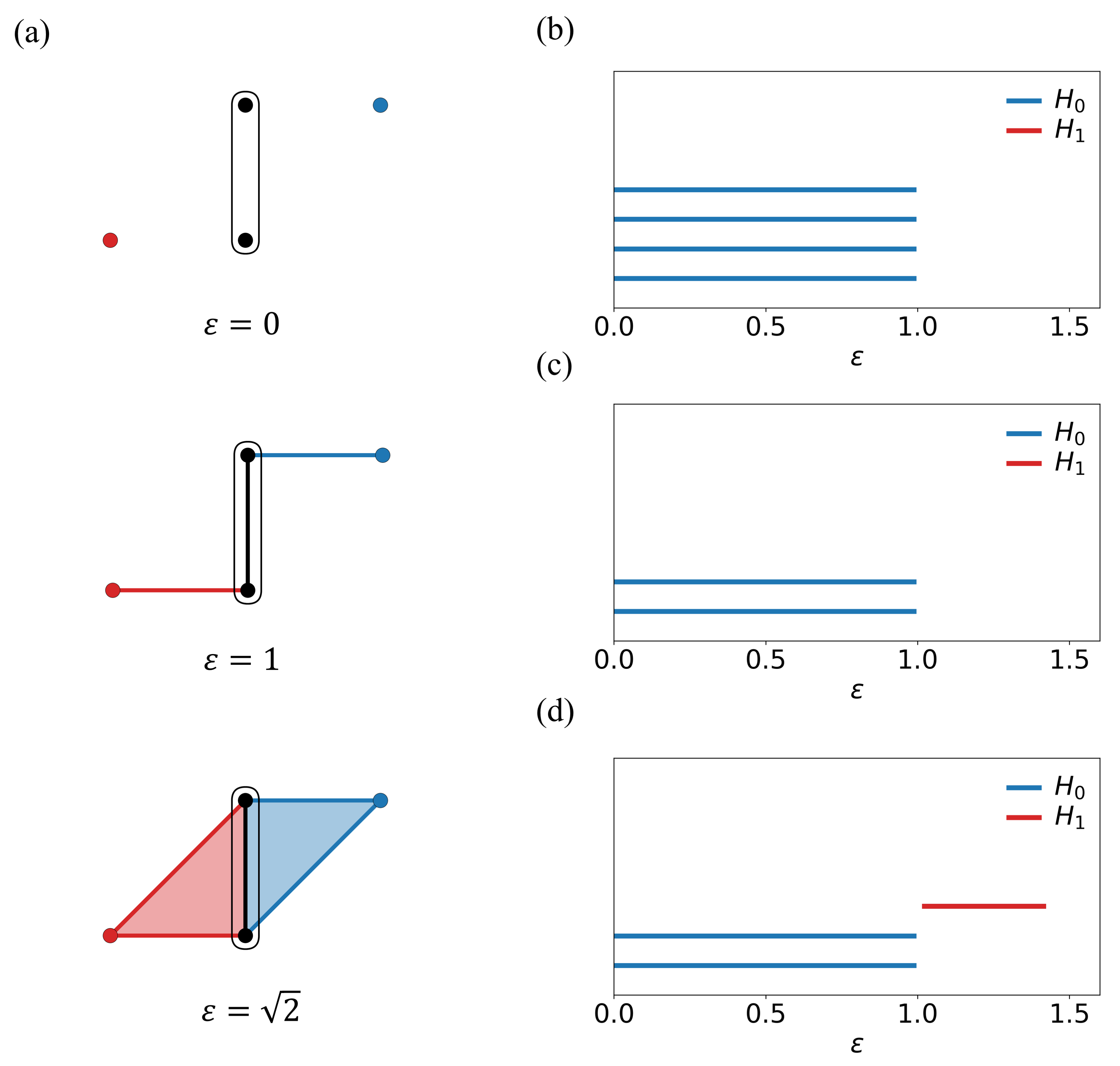}\\
  \caption{(a) The filtration of interaction complexes. The black part represents the intersection component. The red and black points generate the Vietoris-Rips complexes on $X_{1}$, while the filtration of the blue and black points illustrates the Vietoris-Rips complexes on $X_{2}$. The combination of the two interacting filtrations results in the filtration of Vietoris-Rips complexes. (b) The barcode of the Vietoris-Rips complexes on the union of $X_{1}$ and $X_{2}$. (c) The barcode of the Vietoris-Rips complexes on the insection of $X_{1}$ and $X_{2}$. (d) The barcode of the interaction Vietoris-Rips complexes on $\{X_{1},X_{2}\}$.}\label{figure:interaction_filtration}
\end{figure}
The significance of interaction Betti numbers lies in their ability to reflect the interactive relationships within a system, a task not easily achieved with conventional persistent homology alone. In other words, interaction Betti numbers, while reducing computational complexity, focus specifically on the topological properties of relationships between different systems rather than the local topological properties of each individual system.
\end{example}
For other complexes constructed on point set, such as alpha complexes, we can also consider their corresponding interaction complexes and compute the interaction persistent homology.

\subsection{Persistent interaction Laplacians}

Numerous studies have been conducted on the persistent Laplacian \cite{chen2019evolutionary,liu2023algebraic,memoli2022persistent}. In recent years, the persistent Laplacian has also found extensive applications in the fields of biomolecules, drug design, and materials science \cite{chen2022persistent,qiu2023persistent,wang2020persistent,wee2022persistent}. Besides, persistent Laplacians on different objects are also studied \cite{chen2023persistent,wang2023persistent,wei2021persistent}. From now on, we will study the persistent interaction Laplacian and provide some basic calculation examples.

Let $\mathscr{K}_{1}\hookrightarrow \mathscr{K}_{2}\hookrightarrow\cdots \hookrightarrow \mathscr{K}_{n}$ be a filtration of interaction complexes. We endow $IC_{\ast}(\mathscr{K}_{n})$ with the basis as given in Section \ref{section:interaction_laplacian} and regard $IC_{\ast}(\mathscr{K}_{n})$ as an inner product space. By Lemma \ref{lemma:inclusion}, each interaction complex $IC_{\ast}(\mathscr{K}_{i})$ inherits the inner product structure as the subspace of $IC_{\ast}(\mathscr{K}_{n})$.

\begin{lemma}\label{lemma:inclusion}
Let $\{K_{i}\}_{1\leq i\leq n}$ and $\{L_{i}\}_{1\leq i\leq n}$ be interaction complexes. Suppose $\{K_{i}\}_{1\leq i\leq n}$ is a sub interaction space of $\{L_{i}\}_{1\leq i\leq n}$. Then there is a natural inclusion $IC_{\ast}(\{K_{i}\}_{1\leq i\leq n})\hookrightarrow IC_{\ast}(\{L_{i}\}_{1\leq i\leq n})$ of chain complexes.
\end{lemma}
\begin{proof}
By construction, we have the following commutative diagram.
\begin{equation*}
  \xymatrix{
  0\ar@{->}[r]&D_{\ast}(\{K_{i}\}_{1\leq i\leq n}) \ar@{^{(}->}[r]^{\iota} \ar@{^{(}->}[d]^{j|_{D_{\ast}(\{K_{i}\}_{1\leq i\leq n})}}& \bigotimes C_{\ast}(\{K_{i}\}_{1\leq i\leq n})\ar@{->}[r]^{\pi} \ar@{^{(}->}[d]^{j}& IC_{\ast}(\{K_{i}\}_{1\leq i\leq n})\ar@{->}[r]\ar@{->}[d]^{\bar{j}}&0\\
    0\ar@{->}[r]&D_{\ast}(\{L_{i}\}_{1\leq i\leq n})\ar@{^{(}->}[r]^{\widetilde{\iota}}& \bigotimes C_{\ast}(\{L_{i}\}_{1\leq i\leq n})\ar@{->}[r]^{\widetilde{\pi}}& IC_{\ast}(\{L_{i}\}_{1\leq i\leq n})\ar@{->}[r]&0
  }
\end{equation*}
Here, the map $\bar{j}:IC_{\ast}(\{K_{i}\}_{1\leq i\leq n})\to IC_{\ast}(\{L_{i}\}_{1\leq i\leq n})$ is given by $\bar{j}(\bar{x})=j(x)$. We will first prove $\bar{j}$ is injective.
Suppose $\bar{j}(\overline{x})=0$ for some $\overline{x}\in IC_{\ast}(\{K_{i}\}_{1\leq i\leq n})$. There exists an element $x\in \bigotimes C_{\ast}(\{K_{i}\}_{1\leq i\leq n})$ such that $\pi(x)=\overline{x}$. Consequently, $\widetilde{\pi} j(x)=\bar{j}\pi(x)=0$. This implies $j(x)\in \ker\widetilde{\pi}=\im \widetilde{\iota}$. Therefore, we have $j(x)=\widetilde{\iota}(y)$ for some $y\in D_{\ast}(\{L_{i}\}_{1\leq i\leq n})$. Note that the maps $j$ and $\widetilde{\iota}$ are inclusions. Thus we have
\begin{equation*}
  x=y\in \left(\bigotimes C_{\ast}(\{K_{i}\}_{1\leq i\leq n})\right)\cap D_{\ast}(\{L_{i}\}_{1\leq i\leq n})=D_{\ast}(\{K_{i}\}_{1\leq i\leq n}).
\end{equation*}
Consequently, we obtain $\overline{x}=\pi(x)=\pi(\iota(x))=0$. It follows that $\bar{j}$ is injective. The naturality follows from the definition.
\end{proof}

Let $\mathcal{F}:(\mathbb{R},\leq)\to \mathbf{IC_{n}}$ be a persistence interaction complex. For real numbers $a\leq b$, we have an induced morphism of interaction chain complexes
\begin{equation*}
  \theta^{a,b}:IC_{\ast}(\mathcal{F}(a))\to IC_{\ast}(\mathcal{F}(b)).
\end{equation*}
Lemma \ref{lemma:inclusion} ensures that the morphism $\theta^{a,b}$ establishes $IC_{\ast}(\mathcal{F}(a))$ as an inner product subspace of $IC_{\ast}(\mathcal{F}(b))$.
Let $IC_{p+1}^{a,b}=\{x\in IC_{p+1}(\mathcal{F}(b))|dx\in \theta^{a,b}(IC_{p}(\mathcal{F}(a)))\}$. Then we have an inclusion $\iota_{p+1}^{a,b}:IC_{p+1}^{a,b}\hookrightarrow IC_{p+1}(\mathcal{F}(b))$. Let $d_{p+1}^{a,b}:IC_{p+1}^{a,b}\to IC_{p}(\mathcal{F}(a))$ be the composition of the following morphisms
\begin{equation*}
  \xymatrix{
  IC_{p+1}^{a,b}\ar@{->}[r]^-{\iota^{a,b}_{p+1}}&IC_{p+1}(\mathcal{F}(b))\ar@{->}[r]^-{d_{p+1}^{b}}&IC_{p}(\mathcal{F}(b))\ar@{->}[r]^-{(\theta_{p}^{a,b})^{\ast}}&IC_{p}(\mathcal{F}(a)).
  }
\end{equation*}
Then we have a commutative diagram
\begin{equation*}
    \xymatrix@=0.8cm{
  IC_{p+1}(\mathcal{F}(a))\ar@{->}[rr]^{ d_{p+1}^{a}}\ar@{^{(}->}[dd]_{\theta^{a,b}_{p+1}}&&\quad IC_{p}(\mathcal{F}(a))\quad\ar@<0.75ex>[rr]^-{\textcolor[rgb]{0.00,0.07,1.00}{ d_{p}^{a}} } \ar@{^{(}->}[dd]^{\theta^{a,b}_{p}}\ar@<0.75ex>[ld]^-{\textcolor[rgb]{0.00,0.07,1.00}{( d_{p+1}^{a,b})^{\ast}}}&&\quad IC_{p-1}(\mathcal{F}(a))\ar@<0.75ex>[ll]^-{\textcolor[rgb]{0.00,0.07,1.00}{( d_{p}^{a})^{\ast}}}\ar@{^{(}->}[dd]^{\theta^{a,b}_{p-1}}\\
                           &IC_{p+1}^{a,b}\ar@<0.75ex>[ru]^-{\textcolor[rgb]{0.00,0.07,1.00}{ d_{p+1}^{a,b}} }\ar@{^{(}->}[ld]_{\iota_{p+1}^{a,b}}&&&                       \\
  IC_{p+1}(\mathcal{F}(b))\ar@{->}[rr]^{ d_{p+1}^{b}}&&\quad IC_{p}(\mathcal{F}(b))\quad\ar@{->}[rr]^{ d_{p}^{b}}&&\quad IC_{p-1}(\mathcal{F}(b)).
  }
\end{equation*}
The \emph{$p$-th $(a,b)$-persistent interaction Laplacian} $\Delta_{p}^{a,b}:IC_{p}(\mathcal{F}(a))\to IC_{p}(\mathcal{F}(a))$ is defined by
\begin{equation*}
  \Delta_{p}^{a,b}:= d_{p+1}^{a,b}\circ ( d_{p+1}^{a,b})^{\ast}+( d_{p}^{a})^{\ast}\circ d_{p}^{a}.
\end{equation*}
In particular, when $a=b$, the persistent interaction Laplacian $\Delta_{p}^{a,b}$ coincides with the interaction Laplacian $\Delta_{p}^{a}$ on $IC_{p}(\mathcal{F}(a))$. Similarly, the persistent interaction Laplacian operator $\Delta_{p}^{a,b}$ is self-adjoint and non-negative definite. The eigenvalues for $\Delta_{p}^{a,b}$ consist of the spectral of the operator. The smallest positive eigenvalue is the spectral gap, while the second smallest eigenvalue is called the Fiedler vector. The kernel of $\Delta_{p}^{a,b}$ is a subspace of $IC_{p}(\mathcal{F}(a))$, referred to as the $(a,b)$-persistent interaction harmonic space. It is worth noting that the complex $IC_{p}(\mathcal{F}(a))$ has the combinatorial Hodge decomposition with the persistent interaction harmonic space as a direct sum component. More precisely, we have the following proposition.
\begin{proposition}
For any real numbers $a\leq b$, we have $IC_{p}(\mathcal{F}(a))=\mathcal{H}_{p}^{a,b}\oplus \im d_{p+1}^{a,b}\oplus \im (d_{p}^{a})^{\ast}$. Here, the persistent interaction harmonic space $\mathcal{H}_{p}^{a,b}=\ker \Delta_{p}^{a,b}$.
\end{proposition}
\begin{proof}
For any $x\in IC_{p+1}^{a,b}$, we have $d_{p+1}^{b}x=\theta^{a,b}_{p}(y)$ for some $y\in IC_{p}(\mathcal{F}(a))$. It follows that
\begin{equation*}
  d_{p}^{a}  d_{p+1}^{a,b}x=d_{p}^{a}(\theta_{p}^{a,b})^{\ast}d_{p+1}^{b}\iota^{a,b}_{p+1}(x)=d_{p}^{a}(\theta_{p}^{a,b})^{\ast}\theta^{a,b}_{p}(y).
\end{equation*}
Since the injection $\theta^{a,b}$ gives $IC_{\ast}(\mathcal{F}(a))$ the inner product structure inherited from $IC_{\ast}(\mathcal{F}(b))$, we have
\begin{equation*}
  \langle (\theta_{p}^{a,b})^{\ast}\theta^{a,b}_{p}(y),z\rangle=\langle \theta^{a,b}_{p}(y),\theta^{a,b}_{p}(z)\rangle=\langle y,z\rangle
\end{equation*}
for any $z\in IC_{\ast}(\mathcal{F}(a))$. It follows that $(\theta_{p}^{a,b})^{\ast}\theta^{a,b}_{p}(y)=y$. Thus we have
\begin{equation*}
  \theta^{a,b}_{p-1} (d_{p}^{a}  d_{p+1}^{a,b}x)=\theta^{a,b}_{p-1} (d_{p}^{a}y)= d_{p}^{b}\theta^{a,b}_{p}(y)=d_{p}^{b}d_{p+1}^{b}x=0.
\end{equation*}
This shows that $d_{p}^{a}  d_{p+1}^{a,b}=0$. By the algebraic Hodge decomposition theorem, one has the decomposition
\begin{equation*}
  IC_{p}(\mathcal{F}(a))=\mathcal{H}_{p}^{a,b}\oplus \im d_{p+1}^{a,b}\oplus \im (d_{p}^{a})^{\ast},
\end{equation*}
where $\mathcal{H}_{p}^{a,b}=\ker (d_{p+1}^{a,b})^{\ast}\cap \ker d_{p}^{a}$.
\end{proof}
For the simplicial case, the kernel of the persistent Laplacian is isomorphic to the corresponding persistent homology. In the interaction case, the persistent harmonic space associated to $\Delta_{p}^{a,b}$ is also isomorphic to the corresponding persistent interaction homology. By \cite[Theorem 3.6]{liu2023algebraic}, one has the following result.
\begin{proposition}
For any real numbers $a\leq b$, we have a natural isomorphism $\mathcal{H}_{p}^{a,b}\cong H_{p}^{a,b}(\mathcal{F};\mathbb{R})$.
\end{proposition}
The spectrum of the persistent interaction Laplacian $\Delta_{p}^{a,b}$ comprises the eigenvalues of $\Delta_{p}^{a,b}$. The harmonic part of the spectrum corresponds to the zero eigenvalues.
The non-zero eigenvalues of $\Delta_{p}^{a,b}$, referred to as the non-harmonic spectrum of $\Delta_{p}^{a,b}$, capture the geometric information of interaction complexes. Compared to the spectral description of simplicial complexes, which captures the connectivity information between individual simplices, the spectrum of the interaction complex describes the connectivity between interaction pairs. Among these eigenvalues, the smallest positive eigenvalue, denoted by $\tilde{\lambda}_{p}^{a,b}$, holds crucial significance for various applications. Here, we make the convention that if the smallest positive eigenvalue does not exist, we set $\tilde{\lambda}_{p}^{a,b}=0$.

\begin{example}
Example \ref{example:interaction_rips} continued. In this example, we will compute the interaction Laplacians of interaction Vietoris-Rips complexes. The smallest eigenvalues of the interaction Laplacians provide us with crucial features for datesets. By a straightforward calculation, we have that
\begin{equation*}
  \tilde{\lambda}_{0}(\varepsilon)=\left\{
                   \begin{array}{ll}
                     0, & \hbox{$0\leq \varepsilon<1$;} \\
                     3, & \hbox{$1\leq \varepsilon<\sqrt{2}$;} \\
                     4, & \hbox{$\sqrt{2}\leq \varepsilon$.}
                   \end{array}
                 \right.\qquad
\tilde{\lambda}_{1}(\varepsilon)=\left\{
                   \begin{array}{ll}
                     0, & \hbox{$0\leq \varepsilon<1$;} \\
                     3-\sqrt{3}, & \hbox{$1\leq \varepsilon<\sqrt{2}$;} \\
                     2, & \hbox{$\sqrt{2}\leq \varepsilon$.}
                   \end{array}
                 \right.
\end{equation*}
Figure \ref{figure:eigenvalue} shows that curves of smallest positive eigenvalues of interaction Laplacians. As the filtration parameter $\varepsilon$ increases, the smallest positive eigenvalues of interaction Laplacians undergo gradual changes. Specifically, notable variations in the curve are observed at $\varepsilon=1$ and $\varepsilon=\sqrt{2}$. This indicates that the eigenvalues of Laplacians can capture essential information about the key filtration parameters.
\begin{figure}[htbp]
  \centering
  \includegraphics[width=0.9\textwidth]{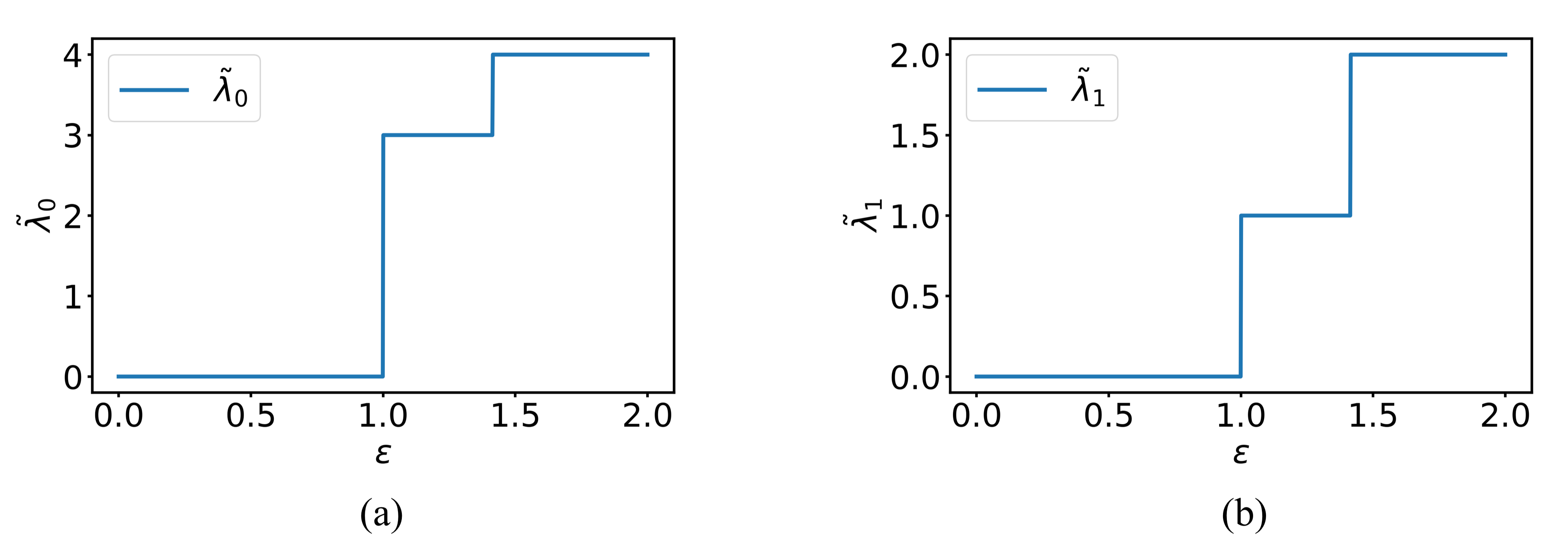}\\
  \caption{(a) The curve of smallest positive eigenvalues of interaction Laplacian at dimension $0$; (b) The curve of smallest positive eigenvalues of interaction Laplacian at dimension $1$.}\label{figure:eigenvalue}
\end{figure}
\begin{equation*}
  \tilde{\lambda}_{1}(\varepsilon)=\left\{
                   \begin{array}{ll}
                     0, & \hbox{$0\leq \varepsilon<1$;} \\
                     2-\sqrt{2}, & \hbox{$1\leq \varepsilon<\sqrt{2}$;} \\
                     2, & \hbox{$\sqrt{2}\leq \varepsilon<\sqrt{5}$.}\\
                     4, & \hbox{$\sqrt{5}\leq \varepsilon$.}\\
                   \end{array}
                 \right.
\end{equation*}
Figure \ref{figure:other_eigenvalues} shows that curves of smallest positive eigenvalues of usual Laplacians of $X_{1}\cup X_{2}$ and $X_{1}\cap X_{2}$. The 0-dimensional eigenvalue curve $\tilde{\lambda}_{0}$ and the 1-dimensional eigenvalue curve $\tilde{\lambda}_{1}$ coincide for the complexes $X_{1}\cup X_{2}$ and $X_{1}\cap X_{2}$, respectively. For the case of the complex $X_{1}\cup X_{2}$, Figures \ref{figure:other_eigenvalues}(a) and (b) do not clearly convey the variation of eigenvalues for the filtration parameter $\varepsilon=1$. Similarly, for the complex $X_{1}\cap X_{2}$, Figures \ref{figure:other_eigenvalues}(c) and (d) do not adequately capture the changes in eigenvalues for the filtration parameter $\varepsilon=\sqrt{2}$.

\begin{figure}[htb!]
  \centering
  \includegraphics[width=0.9\textwidth]{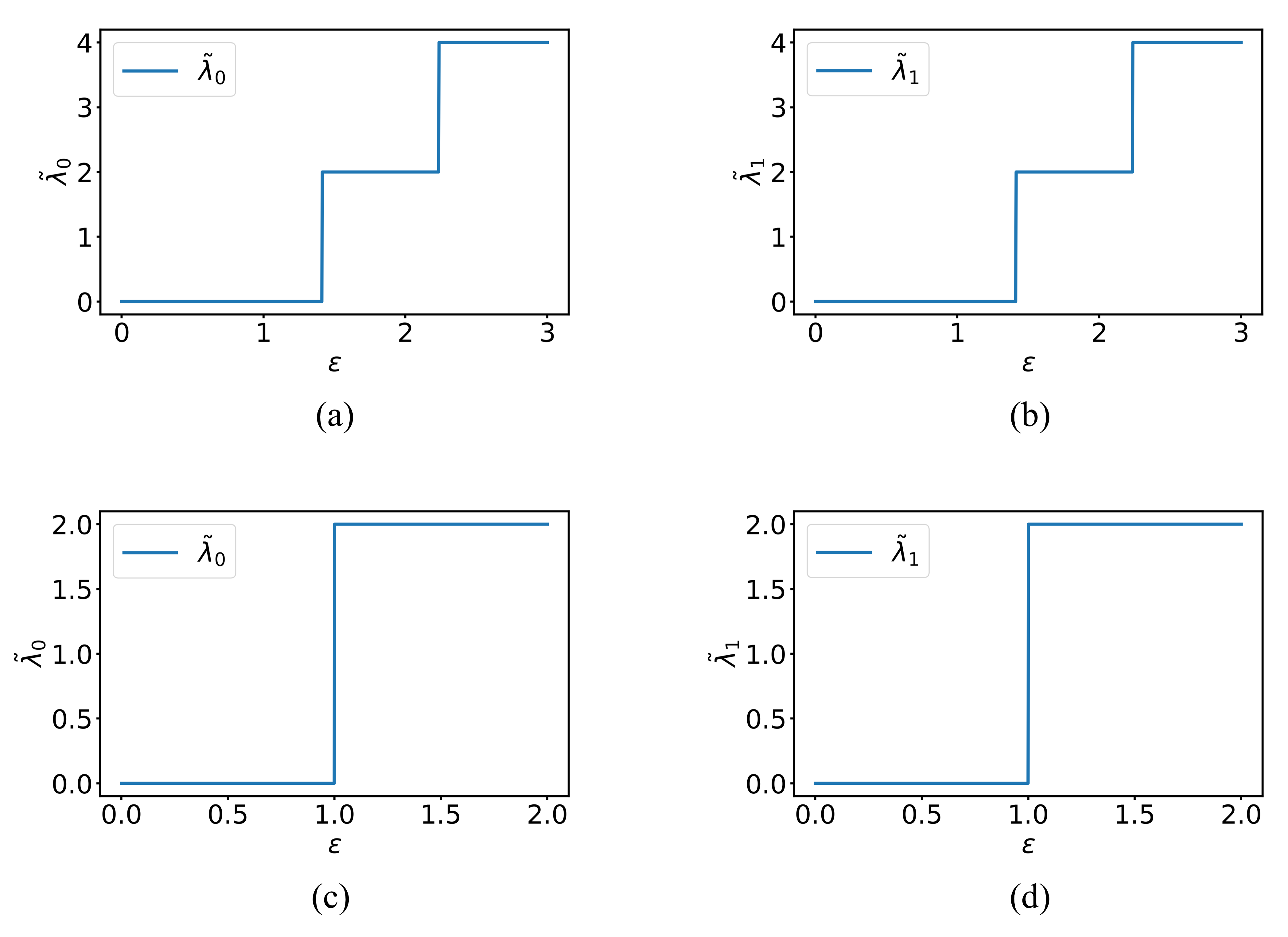}\\
  \caption{(a) The curve of smallest positive eigenvalues of the usual Laplacian of $X_{1}\cup X_{2}$ at dimension $0$; (b) The curve of smallest positive eigenvalues of the usual Laplacian of $X_{1}\cup X_{2}$ at dimension $1$; (c) The curve of smallest positive eigenvalues of the usual Laplacian of $X_{1}\cap X_{2}$ at dimension $0$; (d) The curve of smallest positive eigenvalues of the usual Laplacian of $X_{1}\cap X_{2}$ at dimension $1$.}\label{figure:other_eigenvalues}
\end{figure}
This example highlights the distinctive nature of the interaction Laplacian compared to the conventional Laplacians on simplicial complexes. The interaction Laplacian can offer typical topological and geometric insights into the interaction space, making it a promising tool for various applications.
\end{example}

\section{Applications}\label{section:application}

In this section, we apply persistent interaction homology (PIH) and persistent interaction Laplacians (PILs) to analyze molecular structures. Typically, a molecule comprises various elements, each playing a distinct role in its composition. Moreover, interactions and collaborations exist among different types of atoms within a molecule, implying the presence of internal interactions. Building upon this understanding, we attempt to employ PIH and PILs for element-specific molecular structure analysis. Specifically, interaction Betti numbers and interaction spectral gaps are utilized as topological features to characterize internal interactions within the molecule. We consider two examples: one involving the structural analysis of closo-carboranes. In closo-carboranes, which contain carbon, nitrogen, and hydrogen atoms, the position of carbon atoms is crucial. This prompts us to focus on carbon elements as the interaction components. Through the computation of interaction features, we find that interaction spectral gaps indeed reflect structural information of closo-carboranes molecules. Another application we consider is chlorophyll. In chlorophyll molecules, nitrogen and magnesium atoms are core components responsible for absorbing light energy and converting it into chemical energy. Here, we focus on nitrogen elements as the interaction components and compute their interaction Betti numbers and interaction spectral gaps features. These calculations and applications underscore the vast potential of PIH and PILs.

\subsection{Structural analysis of closo-carboranes}

In this work, we will analyze the structure of closo-carboranes, specifically $\mathrm{C_{2}B_{n-2}H_{n}}$, where $n$ ranges from 5 to 20. Each closo-carborane molecule consists of atoms of three elements: C, B, and H. Typically, the atomic coordinates of these C, B, and H atoms are treated as points, which can be transform into a filtration of Vietoris-Rips complexes. Subsequently, the persistent homology and persistent Laplacian are computed. For each dimension $n$, two curves are obtained: one representing the Betti numbers as a function of the filtration parameter, denoted as $\beta_{n}(t)$, and the other depicting the smallest positive eigenvalue of the Laplacian, also known as the spectral gap, with respect to the filtration parameter, expressed as $\lambda_{n}(t)$. These curves provide insights into the essential topological features and geometric properties of the data points.

One closo-carborane, $\mathrm{C_{2}B_{n-2}H_{n}}$, can be considered as the union of point sets $X_{1}$ and $X_{2}$ corresponding to $\{\mathrm{C_{2}B_{n-2}}\}$ and $\{\mathrm{C_{2}H_{n}}\}$, respectively. Additionally, the intersection of $X_{1}$ and $X_{2}$ consists of points formed by two carbon atoms. As illustrated in Figure \ref{figure:B18_exam}(b), $\mathrm{C_{2}B_{18}H_{20}}$ can be viewed as an interaction between $\{\mathrm{C_{2}B_{18}}\}$ and $\{\mathrm{C_{2}H_{20}}\}$. This interaction reflects the relationships and interactions among atoms of different elements, providing a more nuanced representation than treating all elements as equivalent points. Furthermore, we obtain information from the persistent homology and persistent Laplacians of these distinct sets. In our approach, we treat closo-carboranes as interaction systems, constructing their corresponding interaction Vietoris-Rips complexes. We then compute the interaction homology and interaction Laplacian, resulting in the corresponding interaction Betti curves and interaction spectral gap curves.

\begin{figure}[htbp]
  \centering
  \includegraphics[width=0.9\textwidth]{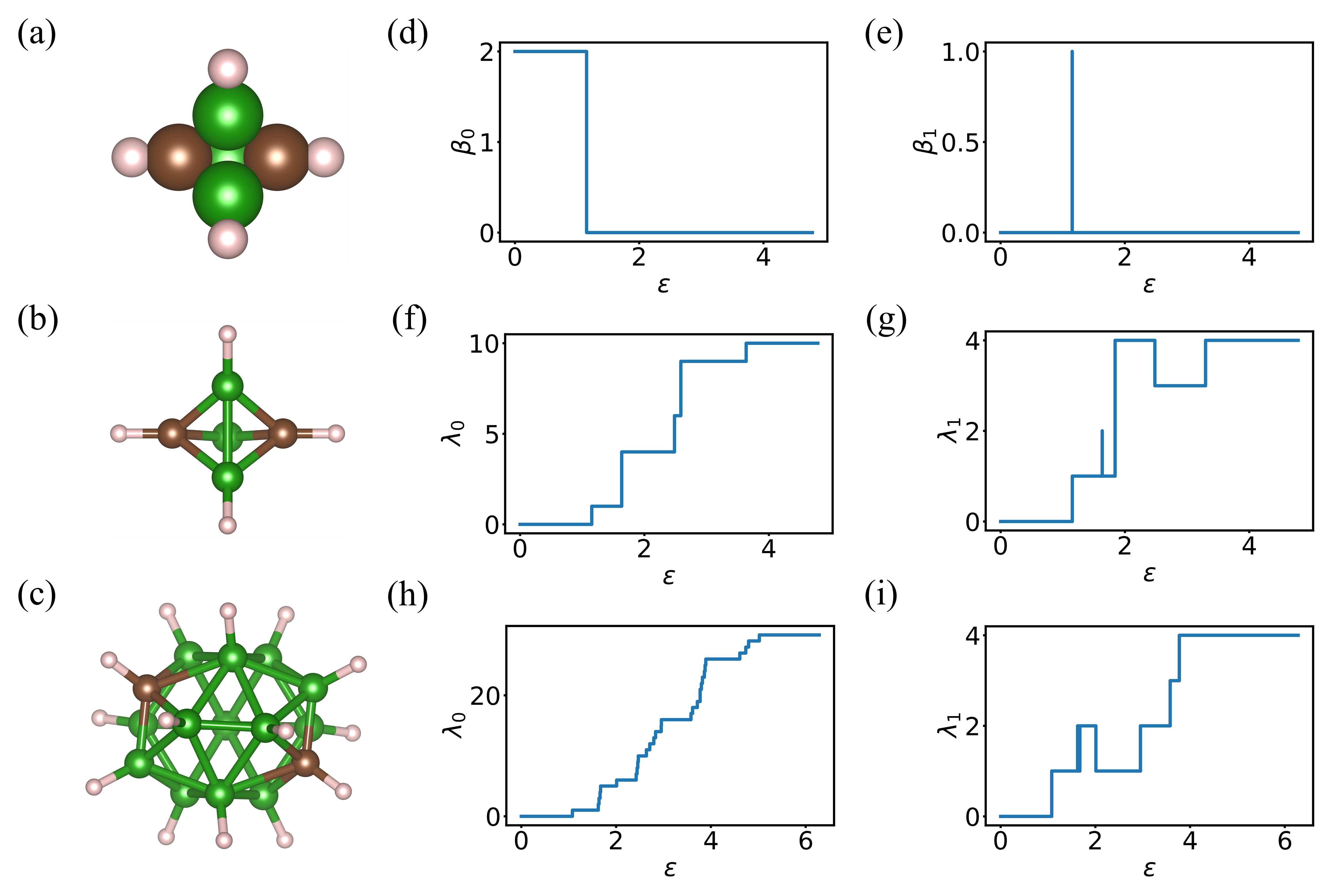}\\
  \caption{(a) The ball representation of the most stability configuration of $\mathrm{C_{2}B_{3}H_{5}}$. (b) The ball-stick representation of the most stability configuration of $\mathrm{C_{2}B_{3}H_{5}}$. (c) The ball-stick representation of the most stability configuration of $\mathrm{C_{2}B_{13}H_{15}}$. (d) The interaction Betti curve of the most stability configuration of $\mathrm{C_{2}B_{3}H_{5}}$ in dimension 0. (e) The interaction Betti curve of the most stability configuration of $\mathrm{C_{2}B_{3}H_{5}}$ in dimension 1. (f) The interaction spectral gap curve of the most stability configuration of $\mathrm{C_{2}B_{3}H_{5}}$ in dimension 0. (g) The interaction spectral gap curve of the most stability configuration of $\mathrm{C_{2}B_{3}H_{5}}$ in dimension 1. (h) The interaction spectral gap curve of the most stability configuration of $\mathrm{C_{2}B_{13}H_{15}}$ in dimension 0. (i) The interaction spectral gap curve of the most stability configuration of $\mathrm{C_{2}B_{13}H_{15}}$ in dimension 1. }\label{figure:molecules}
\end{figure}

Figure \ref{figure:molecules} illustrates the interaction feature curves of the most stable configurations of $\mathrm{C_{2}B_{3}H_{5}}$ and $\mathrm{C_{2}B_{13}H_{15}}$. In our dataset, $\mathrm{C_{2}B_{3}H_{5}}$ exhibits six possible configurations, while $\mathrm{C_{2}B_{13}H_{15}}$ has 135 configurations. Figures \ref{figure:molecules} (b) and (c) depict the bond-stick representations of the most stable configurations for $\mathrm{C_{2}B_{3}H_{5}}$ and $\mathrm{C_{2}B_{13}H_{15}}$, respectively. Figures \ref{figure:molecules} (d) and (e) correspond to the 0-dimensional and 1-dimensional interaction Betti curves for $\mathrm{C_{2}B_{3}H_{5}}$, revealing limited information due to the challenging formation of interaction cycles between $\{\mathrm{C_{2}B_{3}}\}$ and $\{\mathrm{C_{2}H_{5}}\}$. Figures \ref{figure:molecules} (f) and (g) showcase the 0-dimensional and 1-dimensional interaction spectral gaps of $\mathrm{C_{2}B_{3}H_{5}}$, providing more enriched information compared to interaction homology. Meanwhile, the 0-dimensional and 1-dimensional spectral gaps for $\mathrm{C_{2}B_{13}H_{15}}$ are illustrated in Figures \ref{figure:molecules} (h) and (i). This example highlights the rich geometric information inherent in interaction Laplacians, suggesting potential applications.

\begin{figure}[htbp]
  \centering
  \includegraphics[width=0.9\textwidth]{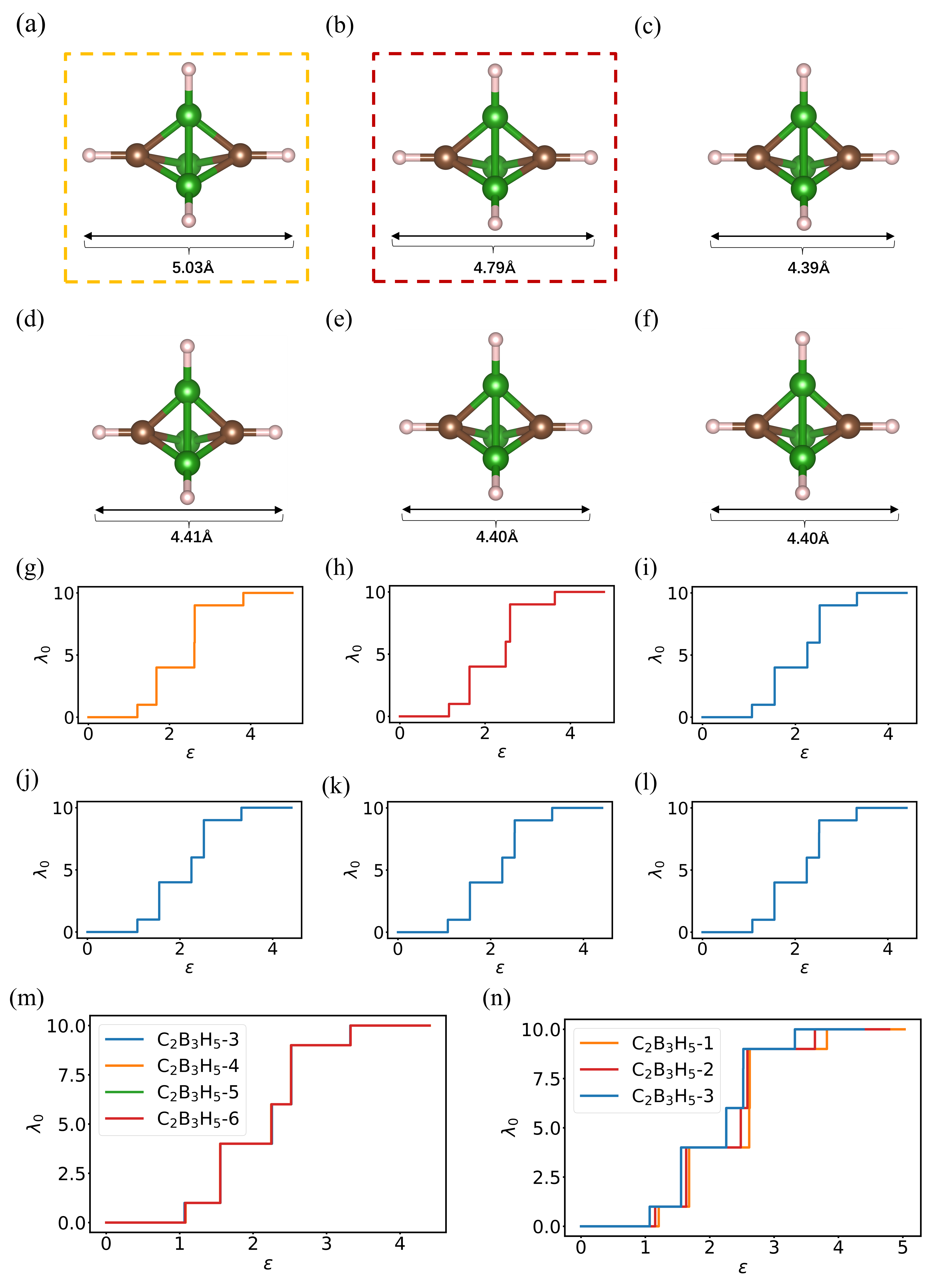}\\
  \caption{(a)-(f) Representation of the six distinct configurations of $\mathrm{C_{2}B_{3}H_{5}}$. These configurations are individually denoted as $\mathrm{C_{2}B_{3}H_{5}}$-1, $\mathrm{C_{2}B_{3}H_{5}}$-2, ... , $\mathrm{C_{2}B_{3}H_{5}}$-6. (g)-(l) Illustration of the interaction spectral gap curves $\tilde{\lambda}_{0}(t)$ for the six different configurations of $\mathrm{C_{2}B_{3}H_{5}}$ in dimension 0. (m) Compilation of interaction spectral gap curves for the configurations $\mathrm{C_{2}B_{3}H_{5}}$-3, $\mathrm{C_{2}B_{3}H_{5}}$-4, $\mathrm{C_{2}B_{3}H_{5}}$-5, and $\mathrm{C_{2}B_{3}H_{5}}$-6. (n) Comparison of the interaction spectral gap curves of the configurations $\mathrm{C_{2}B_{3}H_{5}}$-1, $\mathrm{C_{2}B_{3}H_{5}}$-2, and $\mathrm{C_{2}B_{3}H_{5}}$-3.}\label{figure:B3collection}
\end{figure}

Figure \ref{figure:B3collection} illustrates the six distinct configurations of $\mathrm{C_{2}B_{3}H_{5}}$ and provides the interaction spectral gap curves in dimension 0 for each configuration. Upon careful visual inspection in Figure \ref{figure:B3collection}(a) and (b), we observe subtle differences in the C-H bonds, C-B bonds, B-H bond lengths, and overall shapes between configurations $\mathrm{C_{2}B_{3}H_{5}}$-1 and $\mathrm{C_{2}B_{3}H_{5}}$-2 compared to the other configurations. However, distinctions are harder to discern in Figures \ref{figure:B3collection}(c), (d), (e), (f). The interaction spectral gap curves displayed in Figures \ref{figure:B3collection}(g)-(l) reveal significant differences, particularly for configurations $\mathrm{C_{2}B_{3}H_{5}}$-1 and $\mathrm{C_{2}B_{3}H_{5}}$-2 when compared to the other configurations. As shown in Figure \ref{figure:B3collection}(m), configurations $\mathrm{C_{2}B_{3}H_{5}}$-3, $\mathrm{C_{2}B_{3}H_{5}}$-4, $\mathrm{C_{2}B_{3}H_{5}}$-5, and $\mathrm{C_{2}B_{3}H_{5}}$-6 appear nearly identical. In contrast, Figure \ref{figure:B3collection}(n) indicates significant differences among configurations $\mathrm{C_{2}B_{3}H_{5}}$-1, $\mathrm{C_{2}B_{3}H_{5}}$-2, and $\mathrm{C_{2}B_{3}H_{5}}$-3. Note that $\mathrm{C_{2}B_{3}H_{5}}$-5 and $\mathrm{C_{2}B_{3}H_{5}}$-6 are enantiomers of each other, both possessing an identical enthalpy of formation of -153.810980788 eV \cite{becke2007quantum}. For each configuration $X$ of $\mathrm{C_{2}B_{3}H_{5}}$, we calculate the relative enthalpy of formation using the equation
\begin{equation*}
\Delta E = E(X) - E(\mathrm{C_{2}B_{3}H_{5}}\text{-5}) = E(X) - 153.810980788 .
\end{equation*}
The resulting relative enthalpies of formation are summarized in Table \ref{table:enthalpy}.
\begin{table}[htb!]
\renewcommand{\arraystretch}{1.2}
  \centering
  \caption{The relative enthalpy of formation of the configurations of $\mathrm{C_{2}B_{3}H_{5}}$.}\label{table:enthalpy}
  \begin{tabular}{c|c|c|c|c|c|c}
    \hline
    $\mathrm{C_{2}B_{3}H_{5}}$ & $\mathrm{C_{2}B_{3}H_{5}}$-1 & $\mathrm{C_{2}B_{3}H_{5}}$-2 & $\mathrm{C_{2}B_{3}H_{5}}$-3 & $\mathrm{C_{2}B_{3}H_{5}}$-4 & $\mathrm{C_{2}B_{3}H_{5}}$-5 & $\mathrm{C_{2}B_{3}H_{5}}$-6 \\
    \hline
    $\Delta E$ & 0.078485346 & 0.033148533 & 0.000177721 & 0.00002564 & 0.0 & 0.0 \\
    \hline
  \end{tabular}
\end{table}
From Table \ref{table:enthalpy}, it is evident that the enthalpies of formation for the configurations $\mathrm{C_{2}B_{3}H_{5}}$-3, $\mathrm{C_{2}B_{3}H_{5}}$-4, $\mathrm{C_{2}B_{3}H_{5}}$-5, and $\mathrm{C_{2}B_{3}H_{5}}$-6 are quite close. However, the enthalpies of formation for the configurations $\mathrm{C_{2}B_{3}H_{5}}$-1 and $\mathrm{C_{2}B_{3}H_{5}}$-2 differ significantly from those of the other configurations. Therefore, the information from interaction Laplacians indeed reflects the similarities and differences among various configurationic forms of $\mathrm{C_{2}B_{3}H_{5}}$.

\begin{figure}[htbp]
  \centering
  \includegraphics[width=1\textwidth]{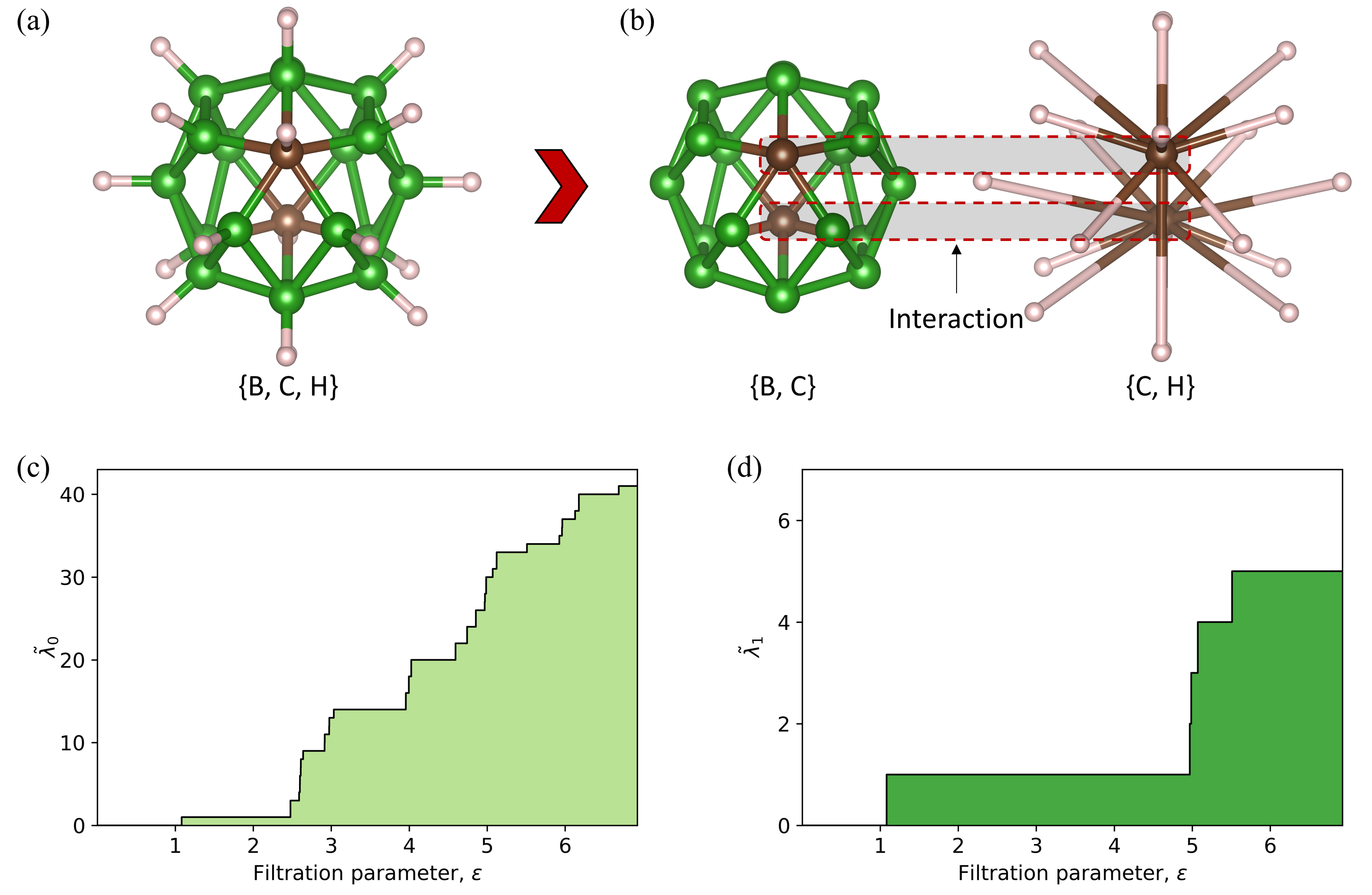}\\
  \caption{(a) The ball-stick representation of $\mathrm{C_{2}B_{18}H_{20}}$. (b) The representation of interaction between $\{\mathrm{C_{2}B_{18}}\}$ and $\{\mathrm{C_{2}H_{20}}\}$ within $\mathrm{C_{2}B_{18}H_{20}}$. (c) The simplicial spectral gap curve of $\mathrm{C_{2}B_{18}H_{20}}$ in dimension 0. (d) The interaction spectral gap curve of $\mathrm{C_{2}B_{18}H_{20}}$ in dimension 0.}\label{figure:B18_exam}
\end{figure}

Compared to classical persistent homology and persistent Laplacian, interaction persistent homology and persistent Laplacian are more capable of reflecting the interactive relationships among different components within a system. Therefore, the information provided by interaction topological features may not necessarily be more comprehensive than that of classical persistent homology and persistent Laplacian. Examining Figure \ref{table:runtime}, when comparing Laplacian information generated by the classic Vietoris-Rips complex with that of the interaction Vietoris-Rips complex, we observe that interaction topological information tends to be simpler. This is because interactions primarily reflect specific interactive information, which is often less globally informative but distinct from overall information. This underscores the potential applications of interaction topology. Particularly, in the case of multi-component systems, where the focus is often on relationships between different parts, the advantages of interaction homology and interaction Laplacian become evident. Furthermore, in our example, the computational efficiency of interaction spectral gaps is notable. We considered closo-carboranes $\mathrm{C_{2}B_{13}H_{15}}$ and $\mathrm{C_{2}B_{18}H_{20}}$, running the computations on our laptop (CPU: AMD Ryzen 5 5600H). As shown in Table \ref{table:runtime}, when comparing the computation time of spectral gap curves between Vietoris-Rips complex and interaction Vietoris-Rips complex, based on the same code design and operating environment, we found that the computational efficiency of interaction spectral gap curves was significantly superior to classic persistent spectral gap curves based on Vietoris-Rips complexes.

\begin{table}[htb!]
\renewcommand{\arraystretch}{1.2}
  \centering
    \caption{The runtime of the curves for the smallest positive eigenvalues of the Laplacians for $\mathrm{C_{2}B_{13}H_{15}}$ and $\mathrm{C_{2}B_{18}H_{20}}$ in dimensions 0 and 1.}\label{table:runtime}
  \begin{tabular}{c|c|c|c|c}
    \hline
    Spectral gaps & $\tilde{\lambda}_{0}$ -- $\mathrm{C_{2}B_{13}H_{15}}$ & $\tilde{\lambda}_{1}$ -- $\mathrm{C_{2}B_{13}H_{15}}$ & $\tilde{\lambda}_{0}$ -- $\mathrm{C_{2}B_{18}H_{20}}$ & $\tilde{\lambda}_{1}$ -- $\mathrm{C_{2}B_{18}H_{20}}$ \\
    \hline
    \hline
    VR complexes & 132.41s &  130.05s & 2866.85s & 30850.31s \\
    Interation VR complexes& 48.23 & 47.99s  & 131.27 & 131.94s \\
    \hline
  \end{tabular}
\end{table}

\subsection{Interaction within chlorophyll}

Chlorophyll is a green pigment found in plants, algae, and some bacteria, and it plays a vital role in photosynthesis, the process by which plants convert light energy into chemical energy. Chlorophyll molecules contain elements such as C, H, O, N, and Mg, with nitrogen (N) and magnesium (Mg) being crucial for photosynthesis. The magnesium atom sits at the core of the chlorophyll molecule, forming its central structure, aiding in the absorption of light energy to excite electron states. Meanwhile, nitrogen primarily stabilizes the molecular structure and plays a vital role in the propagation of electrons and energy. The combined action of nitrogen and magnesium in chlorophyll facilitates the absorption of light energy and its conversion into chemical energy.

Considering chlorophyll-a $\mathrm{C}_{55}\mathrm{H}_{72}\mathrm{O}_{5}\mathrm{N}_{4}\mathrm{Mg}$ and chlorophyll-b $\mathrm{C}_{55}\mathrm{H}_{70}\mathrm{O}_{6}\mathrm{N}_{4}\mathrm{Mg}$ within our mathematical model, we treat the magnesium and nitrogen atoms as a cohesive unit responsible for energy absorption and generation. On the other hand, we regard nitrogen atoms and other chlorophyll structures as a cohesive unit responsible for the complex propagation and utilization of energy. Thus, we delineate two interacting components: $\{\mathrm{C,H,O,N}\}$ and $\{\mathrm{N,Mg}\}$, and compute their persistent interaction homology and persistent interaction Laplacian. Chlorophyll-a and chlorophyll-b are both large molecules, composed of 137 and 136 atoms respectively. Computing their topological properties using the traditional Vietoris-Rips complex and persistent homology would be highly demanding and time-consuming. By employing the interaction persistent homology with persistence parameters ranging from 0{\AA} to 6{\AA}. On our laptop setup (CPU: AMD Ryzen 5 5600H), the computational time is estimated to range from approximately 4000 to 8000 seconds.

As illustrated in Figure \ref{figure:chlorophyll}(a) and (b),  chlorophyll-b closely resembles chlorophyll-a, distinguished by the presence of an additional carbonyl group on one of its rings. The disparities between chlorophyll-a and chlorophyll-b can be discerned through their interaction Laplacians. Specifically, the interaction spectral gap curves of chlorophyll-a and chlorophyll-b exhibit distinct behaviors, as depicted in Figure \ref{figure:chlorophyll}(c) and (d). In Figure \ref{figure:chlorophyll}(e) and (f), we observe the evolution of the interaction Vietoris-Rips complexes corresponding to chlorophyll-a and chlorophyll-b molecules. We find that the interaction Laplacian proves to be more efficient and practical in studying large molecules with multiple elements at certain times. Of course, for a detailed investigation of the interaction topology of molecules, it is necessary to differentiate the coordinates of different elements within large molecules and to select appropriate interaction objects. Under such conditions, we can then compute the interaction topological features.

\begin{figure}[htb!]
  \centering
  \includegraphics[width=1\textwidth]{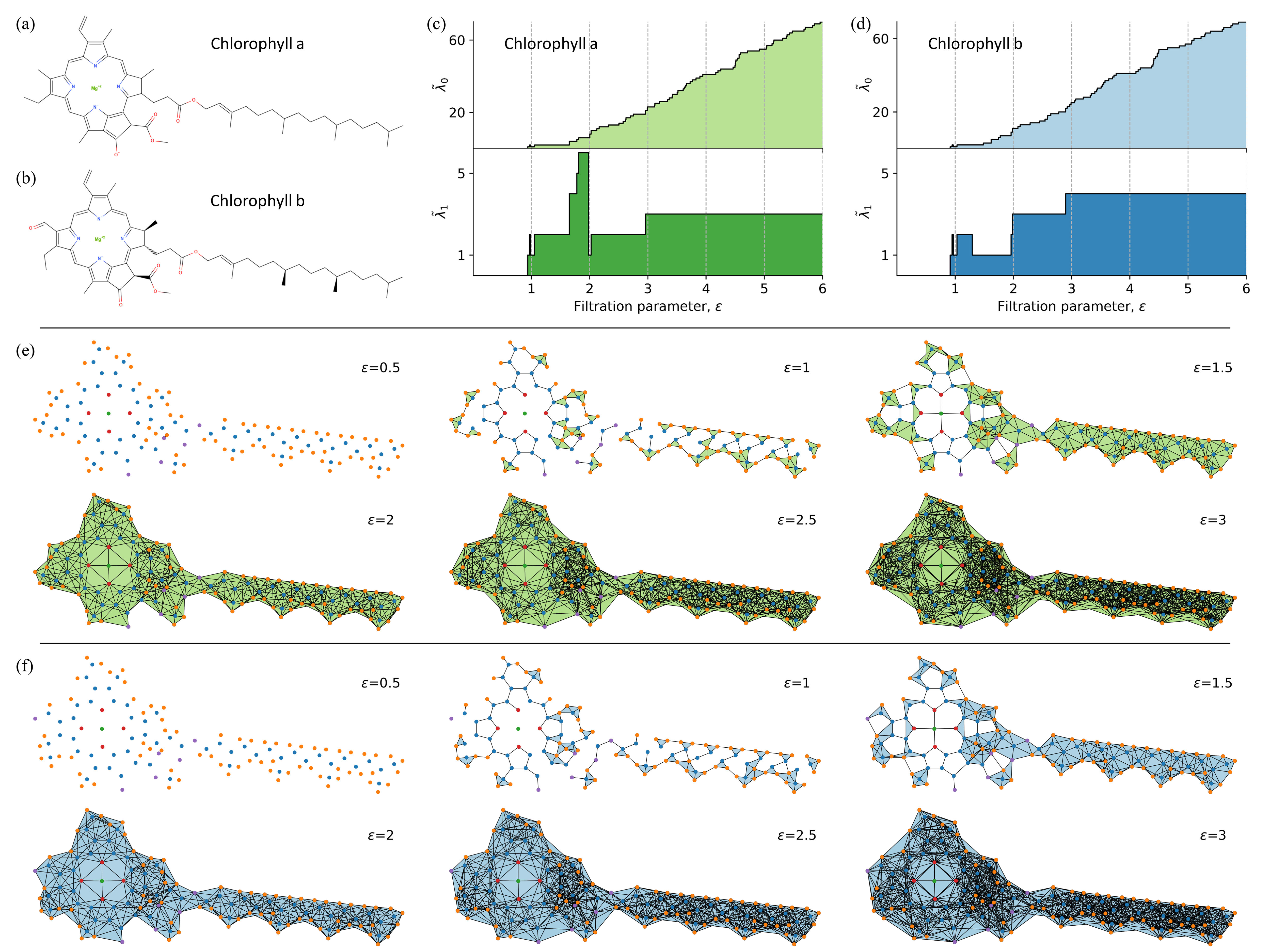}\\
  \caption{(a) The molecular structure representations of chlorophyll-a. (b) The molecular structure representations of chlorophyll-b. (c) The interaction spectral gap curves of chlorophyll-a in dimension 0 and dimension 1. (d) The interaction spectral gap curves of chlorophyll-b in dimension 0 and dimension 1. (e) The filtration of interaction Vietoris-Rips complexes of chlorophyll-a. (f) The filtration of interaction Vietoris-Rips complexes of chlorophyll-b. }\label{figure:chlorophyll}
\end{figure}

\section{Conclusion}

Interaction provides insights into the mutual influence among various elements within a system. Understanding interactions helps elucidate how these components affect each other's behaviors, properties, and functions, often resulting in emergent phenomena or system-level outcomes that cannot be fully explained by considering individual elements in isolation or the point cloud as a whole. However, studying interactions within systems and spaces is often a highly complex task, involving numerous intricate details and sources of interference. From a topological perspective, investigating interaction behaviors is highly worthwhile for two main reasons: topology can grasp the essential elements and holistic features of data; and secondly, element-specific topological features are conveniently computable and possess strong robustness. For these reasons, we focus on applying theories related to interaction topology in data science.

In this work, we present the concepts of persistent interaction homology (PIH) and persistent interaction Laplacian (PIL). The interaction topology focuses on the homotopy types of intersections between subspaces within a space, along with their corresponding homology invariants. We begin by establishing the concept of the interaction Laplacian, wherein its kernel space is isomorphic to the interaction homology. Building upon the foundation of interaction topology, we introduce PIH and PILs as innovative methods within the field of topological data analysis. Furthermore, we also demonstrate the stability of persistent interaction homology as a persistent module. For finite point sets in Euclidean spaces, we provide the construction of interaction Vietoris-Rips complexes, enabling the computation of persistent interaction homology and persistent interaction Laplacians for finite point sets. The proposed persistent interaction topology can be utilized to extract  element-specific topological information, which is essential to the success of TDA \cite{cang2017topologynet}.
In practical applications, we utilize PIH and PILs for extracting topological features from closo-carboranes and chlorophyll molecules, demonstrating the utility of the propsoed methods.

In the future, we envision further advancements in the methodology of PIH and PILs. On a theoretical part, PIH and PILs currently reside at a conceptual stage, awaiting deeper exploration of the underlying mathematical principles and ideas. In applications, we aspire to see the use of PIH and PILs in data science,  tackling specific challenges in physics, materials science, molecular biology, and beyond. We believe that PIH and PILs will play a significant role in a wide range of fields.

\section*{Data and code availability}

The data and source coda used in this work is publicly available in the Github repository: \url{https://github.com/WeilabMSU/InteractionTop}.

\section*{Acknowledgments}
This work was supported in part by NIH grants  R01GM126189 and  R01AI164266, NSF grants DMS-2052983,  DMS-1761320, and IIS-1900473,  NASA grant 80NSSC21M0023,  MSU Foundation,  Bristol-Myers Squibb 65109, and Pfizer.

\bibliographystyle{plain}  
\bibliography{Reference}

\begin{thebibliography}{10}

\bibitem{alexandroff1928allgemeinen}
Paul Alexandroff.
\newblock {\"U}ber den allgemeinen dimensionsbegriff und seine beziehungen zur
  elementaren geometrischen anschauung.
\newblock {\em Mathematische Annalen}, 98(1):617--635, 1928.

\bibitem{ameneyro2024quantum}
Bernardo Ameneyro, Vasileios Maroulas, and George Siopsis.
\newblock Quantum persistent homology.
\newblock {\em Journal of Applied and Computational Topology}, pages 1--20,
  2024.

\bibitem{bauer2020persistence}
Ulrich Bauer and Michael Lesnick.
\newblock Persistence diagrams as diagrams: A categorification of the stability
  theorem.
\newblock In {\em Topological Data Analysis: The Abel Symposium 2018}, pages
  67--96. Springer, 2020.

\bibitem{becke2007quantum}
Axel Becke.
\newblock {\em The quantum theory of atoms in molecules: from solid state to
  DNA and drug design}.
\newblock John Wiley \& Sons, 2007.

\bibitem{bott1982differential}
Raoul Bott, Loring~W Tu, et~al.
\newblock {\em Differential forms in algebraic topology}, volume~82.
\newblock Springer, 1982.

\bibitem{bubenik2014categorification}
Peter Bubenik and Jonathan~A Scott.
\newblock Categorification of persistent homology.
\newblock {\em Discrete \& Computational Geometry}, 51(3):600--627, 2014.

\bibitem{cang2017topologynet}
Zixuan Cang and Guo-Wei Wei.
\newblock Topologynet: Topology based deep convolutional and multi-task neural
  networks for biomolecular property predictions.
\newblock {\em PLoS computational biology}, 13(7):e1005690, 2017.

\bibitem{carlsson2004persistence}
Gunnar Carlsson, Afra Zomorodian, Anne Collins, and Leonidas Guibas.
\newblock Persistence barcodes for shapes.
\newblock In {\em Proceedings of the 2004 Eurographics/ACM SIGGRAPH symposium
  on Geometry processing}, pages 124--135, 2004.

\bibitem{chazal2009proximity}
Fr{\'e}d{\'e}ric Chazal, David Cohen-Steiner, Marc Glisse, Leonidas~J Guibas,
  and Steve~Y Oudot.
\newblock Proximity of persistence modules and their diagrams.
\newblock In {\em Proceedings of the twenty-fifth annual symposium on
  Computational geometry}, pages 237--246, 2009.

\bibitem{chen2023persistent}
Dong Chen, Jian Liu, Jie Wu, and Guo-Wei Wei.
\newblock Persistent hyperdigraph homology and persistent hyperdigraph
  {L}aplacians.
\newblock {\em Foundations of Data Science}, 5(4):558--588.

\bibitem{chen2022persistent}
Jiahui Chen, Yuchi Qiu, Rui Wang, and Guo-Wei Wei.
\newblock {Persistent Laplacian projected Omicron BA. 4 and BA. 5 to become new
  dominating variants}.
\newblock {\em Computers in Biology and Medicine}, 151:106262, 2022.

\bibitem{chen2019evolutionary}
Jiahui Chen, Rundong Zhao, Yiying Tong, and Guo-Wei Wei.
\newblock Evolutionary de {Rham-Hodge} method.
\newblock {\em arXiv preprint arXiv:1912.12388}, 2019.

\bibitem{edelsbrunner2008persistent}
Herbert Edelsbrunner, John Harer, et~al.
\newblock Persistent homology-a survey.
\newblock 2008.

\bibitem{edelsbrunner2015persistent}
Herbert Edelsbrunner, Grzegorz Jab{\l}o{\'n}ski, and Marian Mrozek.
\newblock The persistent homology of a self-map.
\newblock {\em Foundations of Computational Mathematics}, 15:1213--1244, 2015.

\bibitem{eilenberg2015foundations}
Samuel Eilenberg and Norman Steenrod.
\newblock {\em Foundations of algebraic topology}, volume 2193.
\newblock Princeton University Press, 2015.

\bibitem{knill2018cohomology}
Oliver Knill.
\newblock The cohomology for {W}u characteristics.
\newblock {\em arXiv preprint arXiv:1803.06788}, 2018.

\bibitem{liu2023interaction}
Jian Liu, Dong Chen, and Guo-Wei Wei.
\newblock Interaction homotopy and interaction homology.
\newblock {\em arXiv preprint arXiv:2311.16322}, 2023.

\bibitem{liu2023algebraic}
Jian Liu, Jingyan Li, and Jie Wu.
\newblock The algebraic stability for persistent {L}aplacians.
\newblock {\em arXiv preprint arXiv:2302.03902}, 2023.

\bibitem{memoli2022persistent}
Facundo M{\'e}moli, Zhengchao Wan, and Yusu Wang.
\newblock Persistent laplacians: Properties, algorithms and implications.
\newblock {\em SIAM Journal on Mathematics of Data Science}, 4(2):858--884,
  2022.

\bibitem{qiu2023persistent}
Yuchi Qiu and Guo-Wei Wei.
\newblock Persistent spectral theory-guided protein engineering.
\newblock {\em Nature computational science}, 3(2):149--163, 2023.

\bibitem{shen2023persistent}
Li~Shen, Jian Liu, and Guo-Wei Wei.
\newblock Persistent {M}ayer homology and persistent {M}ayer {L}aplacian.
\newblock {\em arXiv preprint arXiv:2312.01268}, 2023.

\bibitem{suwayyid2024persistent}
Faisal Suwayyid and Guo-Wei Wei.
\newblock Persistent dirac of paths on digraphs and hypergraphs.
\newblock {\em Foundations of Data Science}, 10.3934/fods.2024001, 2024.

\bibitem{wang2019persistent}
Rui Wang, Duc~Duy Nguyen, and Guo-Wei Wei.
\newblock Persistent spectral graph.
\newblock {\em arXiv preprint arXiv:1912.04135}, 2019.

\bibitem{wang2020persistent}
Rui Wang, Duc~Duy Nguyen, and Guo-Wei Wei.
\newblock Persistent spectral graph.
\newblock {\em International journal for numerical methods in biomedical
  engineering}, 36(9):e3376, 2020.

\bibitem{wang2023persistent}
Rui Wang and Guo-Wei Wei.
\newblock Persistent path {L}aplacian.
\newblock {\em Foundations of data science}, 5(1):26, 2023.

\bibitem{wee2023persistent}
JunJie Wee, Ginestra Bianconi, and Kelin Xia.
\newblock Persistent dirac for molecular representation.
\newblock {\em Scientific Reports}, 13(1):11183, 2023.

\bibitem{wee2022persistent}
JunJie Wee and Kelin Xia.
\newblock Persistent spectral based ensemble learning ({PerSpect-EL}) for
  protein--protein binding affinity prediction.
\newblock {\em Briefings in Bioinformatics}, 23(2):bbac024, 2022.

\bibitem{wei2023persistent}
Xiaoqi Wei, Jiahui Chen, and Guo-Wei Wei.
\newblock Persistent topological laplacian analysis of sars-cov-2 variants.
\newblock {\em Journal of computational biophysics and chemistry}, 22(5):569,
  2023.

\bibitem{wei2021persistent}
Xiaoqi Wei and Guo-Wei Wei.
\newblock Persistent sheaf {L}aplacians.
\newblock {\em arXiv preprint arXiv:2112.10906}, 2021.

\bibitem{wells1980differential}
Raymond~O'Neil Wells and Oscar Garc{\'\i}a-Prada.
\newblock {\em Differential analysis on complex manifolds}, volume 21980.
\newblock Springer New York, 1980.

\bibitem{zomorodian2004computing}
Afra Zomorodian and Gunnar Carlsson.
\newblock Computing persistent homology.
\newblock In {\em Proceedings of the twentieth annual symposium on
  Computational geometry}, pages 347--356, 2004.

\end{thebibliography}

\end{document}